\documentclass[11pt]{article}
\pdfoutput=1
\usepackage[all,cmtip]{xy}
\usepackage{amssymb}
\usepackage{amsmath}
\usepackage{ mathrsfs }     %For calligraphic T's
\usepackage{verbatim}   %For \comment{...}
\usepackage{graphicx}   %includes xypic, I think?
\usepackage{url}
\usepackage{hyperref}

\newtheorem{theorem}{Theorem}[section]
\newtheorem{lemma}[theorem]{Lemma}
\newtheorem{proposition}[theorem]{Proposition}

\newtheorem{definition}[theorem]{Definition}

\newcommand{\qed}{\hfill \ensuremath{\Box}}

\newenvironment{proof}[1][Proof]{\begin{trivlist}
\item[\hskip \labelsep {\bfseries #1}]}{\end{trivlist}}

\def\rs{\widehat{\mathbb{C}}}
\def\phibar{\overline{\phi}}
\def\psibar{\overline{\psi}}
\def\badset{\mathcal{B}_{\phi}}
\def\extrabadset{\mathcal{B}_{\overline{\phi}}}
\def\Qbar{\overline{\mathbb{Q}}}

\begin{document}

\title{Boundary Values of the Thurston Pullback Map}
\author{Russell Lodge\\
  Jacobs University\\
  Bremen, Germany\\
  \texttt{r.lodge@jacobs-university.de}}
\date{\today}
\maketitle

\begin{abstract}
For any Thurston map with exactly four postcritical points, we present an algorithm to compute the Weil-Petersson boundary values of the corresponding Thurston pullback map.  This procedure is carried out for the Thurston map $f(z)=\frac{3z^2}{2z^3+1}$ originally studied by Buff, et al.  The dynamics of this boundary map are investigated and used to solve the analogue of Hubbard's Twisted Rabbit problem for $f$.
\end{abstract}

%\begin{abstract}
%The Thurston characterization and rigidity theorem gives a beautiful describes when a postcritically finite topological branched cover mapping the 2-sphere to itself is Thurston equivalent to a rational function.  Thurston equivalence remains mysterious, but great strides were taken when Bartholdi and Nekrashevych used iterated monodromy groups to solve the ``Twisted Rabbit Problem" which sought to identify the Thurston class of the rabbit polynomial composed with an arbitrary Dehn twist.  The pullback on curves is here computed for $f(z)=\frac{3z^2}{2z^3+1}$, and its dynamical properties yield a promising invariant for Thurston equivalence.  This invariant is used to give the only existing solution to the twisting problem in the non-polynomial case. The pulllback on curves is also used to give a description of the extended Thurston pullback map on the Weil-Petersson boundary of Teichm\"uller which was shown to exist by Selinger.
%\end{abstract}

\tableofcontents
\section{Introduction}

In the early 1980's, Douady and Hubbard highlighted a close relationship between combinatorics and complex dynamics when they used trees to model the Julia set of postcritically finite polynomials \cite{BKS}.  Much is known about the relationship between combinatorics and Julia sets for complex polynomials; however, non-polynomial rational functions are another matter.

\subsection{Thurston's Theorem}

To produce combinatorial models of postcritically finite rational functions, it is common to define a topological model of a function with the desired combinatorics and then determine if it is actually equivalent to a rational function.  Thus, one defines a Thurston map to be a postcritically finite orientation-preserving branched cover from the two-sphere to itself.  A number of frequently used methods exist to produce these topological examples of Thurston maps, and the following two make appearances in this paper.  Finite subdivision rules of the two-sphere have given a wealth of easily described examples that have been studied extensively in the recent past \cite{CFP}.  Another way to produce examples of Thurston maps is by mating two postcritically finite polynomials of the same degree, yielding a Thurston map of equal degree \cite{Me11,Mi00}.

It is then natural to ask if these Thurston maps actually model rational functions up to an appropriate notion of combinatorial equivalence; Thurston's theorem characterizes which maps are equivalent to rational functions (Thm \ref{thm:Thurston}).  We focus here on the generic case of a Thurston map that is not equivalent to one of the well-understood Latt\'es examples.  A Thurston map is called obstructed if it is not equivalent to a rational function.  These obstructed maps are characterized by a special family of simple closed curves in the two-sphere that behave a certain way under preimage.  Furthermore, the theorem establishes a rigidity result---if the Thurston map is equivalent to a rational function, this rational function is essentially unique.  The proof of Thurston's theorem is given in \cite{DH93} using iteration of an analytic map on Teichm\"uller space called Thurston's pullback map.  It is shown that the pullback map corresponding to a Thurston map has a fixed point if and only if the Thurston map is equivalent to a rational map.  It is worth noting that the mapping properties of simple closed curves play an important role in the three other groundbreaking theorems proved by Thurston relating geometry and topology, as noted by Hubbard \cite{H06}.

\subsection{Boundary Values of the Pullback Map}

Thurston's pullback map is complicated---in all known nontrivial examples it is infinite-to-one, and recent years have seen a number of papers devoted to the study of its properties \cite{BEKP,KPS,Sel}.  Selinger showed that the map extends to the Weil-Peterson boundary, but few explicit computations of the boundary values have been made with the notable exception of one class of maps.  A Thurston map is called \textit{nearly Euclidean} if all critical points have local degree two and the postcritical set has exactly four points; an algorithm exists to compute the boundary values of the pullback map \cite{CFPP2}.  Using different methods, we present an algorithm that computes the boundary values for any Thurston map with four postcritical points.  The tradeoff for this greater generality is that a virtual endomorphism must be computed, which can be challenging.

Let $F$ be a Thurston map with postcritical set $P_F$ containing exactly four points.  A simple closed curve $\gamma\subset\widehat{\mathbb{C}}\setminus P_F$ is called \textit{essential} if both components bounded by $\gamma$ contain exactly two post-critical points.  The collection of homotopy classes of essential curves in $\widehat{\mathbb{C}}\setminus P_F$ can be put in bijective correspondence with the extended rational numbers $\Qbar:=\mathbb{Q}\cup\{\frac{1}{0}\}$.  Furthermore, it is a standard fact that Teichm\"uller space can be identified with the upper half-plane in such a way that the Weil-Petersson completion is obtained by adding the rational points $\frac{p}{q}+0i$ along with the point $\frac{1}{0}$ to the boundary of the upper half-plane where each point corresponds to collapsing the sphere with four marked points along the curve corresponding to $\frac{p}{q}$.

Selinger proved that the extended Thurston pullback map sends the point $\frac{p}{q}$ to the point $\frac{p'}{q'}$ if the preimage under $F$ of the curve corresponding to $\frac{p}{q}$ has an essential component in the homotopy class of the curve corresponding to $\frac{p'}{q'}$ (if there is no essential component, the point maps to the interior of Teichm\"uller space) \cite{Sel}.  Thus, to understand the boundary values of the pullback map under the Weil-Petersson completion, one can simply understand the essential preimages of curves under $f$ and vice versa.  In this paper, these curve preimages are computed using the algebraic machinery of virtual endomorphisms drawn from \cite{BN06}.  This kind of curve preimage computation has already been made for some quadratic polynomials \cite{Pil10}.  When the boundary $\overline{\mathbb{Q}}$ maps to itself, the function $\sigma_F:\overline{\mathbb{Q}}\longrightarrow\overline{\mathbb{Q}}$ is defined by $\sigma_F(\frac{p}{q})=\frac{p'}{q'}$.  

A number of choices must be made to carry out these computations, and one goal of this paper is to present a standard method for making these choices.  To this end, explicit bijections are established between free homotopy classes of unoriented essential curves in the dynamical plane, elements of $\overline{\mathbb{Q}}$, even continued fraction expansions of such elements, elements of the Weil-Petersson boundary, primitive ``right-hand" parabolic elements of $P\Gamma(2)$, maximal parabolic subgroups of $P\Gamma(2)$, primitive ``right-hand" peripheral elements in the fundamental group of moduli space, and right-hand Dehn twists.  These bijections are all natural, and embody the ``fivefold way" for general Thurston maps \cite[Thm 2.1]{KPS}.

\subsection{Boundary Values for an Example}

For the map $f(z)=\frac{3z^2}{2z^3+1}$ which has four postcritical points, it will be shown that $\sigma_f:\Qbar\rightarrow\Qbar$ is computed by transforming the even continued fraction expansions of the rational numbers in a prescribed way.  This particular $f$ was studied in \cite{BEKP}, and one finds there an image of the pullback map for $f$ generated by X. Buff.  It was also shown that the Thurston pullback map associated to $f$ is surjective, and to date this is essentially the only such example up to some nondynamical equivalence.  The following result about the dynamics of $\sigma_f$ will be proven in Section \ref{SigmaProps}.
\bigskip

\noindent\textbf{Theorem} Let $\frac{p}{q}$ be a reduced fraction.  Then under iteration of $\sigma_f$, $\frac{p}{q}$ lands either on the two-cycle $\frac{0}{1}\leftrightarrow\frac{1}{0}$ or on the fixed point $-\frac{1}{1}$.  More precisely, $\frac{p}{q}$ lands on $-\frac{1}{1}$ if and only if $p$ and $q$ are odd. 

\bigskip

Thus $\sigma_f$ has a finite global attractor, meaning that every point $\frac{p}{q}\in \overline{\mathbb{Q}}$ lands in some finite set under iteration.  It will also be proven that $\sigma_f$ is surjective and infinite-to-one at every point.

\subsection{The Twisting Problem}
 In addition to being useful in their own right, these explicit boundary computations provide a useful invariant for the notion of equivalence used in Thurston's theorem (Defn \ref{ThurstonEquivalence}).  Thurston equivalence remains mysterious, and many natural problems related to it are quite difficult to solve.  For example, the ``Twisted Rabbit problem" is to determine whether the rabbit polynomial composed with an arbitrary Dehn twist is equivalent to the rabbit, corabbit, or airplane polynomial.   Pilgrim described the lack of solution to the Twisted Rabbit Problem for over a decade a ``humbling reminder" of the lack of suitable invariants \cite{Pil10}.  The solution came in the work of Bartholdi and Nekrashevych which introduced the permutational biset and the iterated monodromy group as invariants of Thurston equivalence for quadratic polynomials \cite{BN06}.  Their machinery effectively reduced the topological question of determining Thurston class to the algebraic question of determining nuclei of iterated monodromy groups.  Though these methods work well in the setting of quadratic polynomials, it is unclear how to generalize them to the case of rational functions of higher degree.

The solution presented here to the twisting problem for $f(z)=\frac{3z^2}{2z^3+1}$ uses the dynamics of the boundary values of the pullback map as an invariant for Thurston equivalence.  The pure mapping class group of $\rs$ with marked points $P_f$ is the set of homeomorphisms that fix $P_f$ pointwise modulo isotopy fixing $P_f$, and is denoted by $\text{PMCG}(\widehat{\mathbb{C}},P_f)$.  Since $|P_f|=4$, this group has two generators $\alpha$ and $\beta$ which are chosen explicitly.  In the spirit of \cite{BN06}, the virtual endomorphism is extended to a function $\overline{\psi}:\text{PMCG}(\widehat{\mathbb{C}},P_f)\longrightarrow
\text{PMCG}(\widehat{\mathbb{C}},P_f)$ so that $g\circ f$ is Thurston equivalent to $\overline{\psi}(g) \circ f$.  The following theorem from \cite{Lod} describes the dynamics of $\overline{\psi}$.

\bigskip

\noindent\textbf{Theorem:} For any $g\in\text{PMCG}(\widehat{\mathbb{C}},P_f)$ there is a positive number $N$ so that $\overline{\psi}^{\circ n}(g)\in \mathfrak{M}$ for all $n>N$ where
\begin{center}{$\mathfrak{M}=\{e,\beta,\alpha^{-1},\alpha^2\beta^{-1},
\alpha^{-1}\beta\alpha^{-1},\alpha\beta^{-1},\beta^2\}\cup\{\alpha(\beta\alpha)^k:k\in\mathbb{Z}\}$}
\end{center}

Thus the twisting problem reduces to identifying the Thurston class of each element of $\mathfrak{M}$ applied to $f$.  It is shown that composing each element of $\{\alpha(\beta\alpha)^k:k\in\mathbb{Z}\}$ with $f$ produces a one-parameter family of obstructed and pairwise inequivalent maps.  If on the other hand $h\in\{e,\beta,\alpha^{-1},\alpha^2\beta^{-1},
\alpha^{-1}\beta\alpha^{-1},\alpha\beta^{-1},\beta^2\}$, the Thurston class of $h\circ f$ is not immediately clear.   It is possible to show that $h\circ f$ is unobstructed using wreath recursions and that it must be equivalent to $f$ itself or to a second fixed rational function $g$.  However, repeated attempts with known methods were unable to determine whether it was equivalent to $f$ or $g$.  The needed invariant was found, though, when the pullback on curves for $g$  was shown to have two distinct two-cycles.  Thus, if the pullback of $h\circ f$ on curves has two two-cycles, it cannot be equivalent to $f$ because the finite global attractor only has one two-cycle.  This fact along with the fact that $g$ is not a mating allows for a complete solution to the twisting problem for $f$, and it brings to light a valuable invariant of Thurston equivalence.

\subsection{Outline}

Section \ref{sec:npoints} introduces Thurston theory for the sphere with a finite number of marked points.  Correspondences on moduli space are discussed, and virtual endomorphisms are defined on the pure mapping class group of the dynamical plane and the fundamental group of moduli space as in \cite{BN06}.  The Reidemeister-Schreier algorithm is used to find generators for the domain of the latter virtual endomorphism.  Also, given any element of this domain of definition in terms of the generators of the fundamental group, the algorithm rewrites this element in terms of the generators of the domain.  This facilitates the virtual endomorphism computations made in later sections.  Iterated monodromy groups and wreath recursions are discussed, and a summary is given of known results about the pullback on curves and solutions to twisted polynomial problems.

Section \ref{sec:FourPostcriticalPoints} specializes to the case of four marked points.  An explicit procedure is described to identify maximal parabolic subgroups of P$\Gamma(2)$, maximal parabolic subgroups of the fundamental group of the thrice-punctured sphere, points in the Weil-Petersson boundary, essential curves in the sphere with four punctures,  primitive right Dehn twists, and the extended rational numbers $\overline{\mathbb{Q}}$.  Section \ref{sec:slope} defines the notion of slope for essential curves in the four-punctured sphere by lifting to the torus double cover.  This method of assigning extended rational numbers to essential curves follows \cite{CFPP2} and is essentially different from the method in the previous section.

Section \ref{fProperties} presents a study of the postcritically finite rational function $f(z)=\frac{3z^2}{2z^3+1}$.  A convenient combinatorial model for this $f$ is presented, and the virtual endomorphisms on the dynamical plane and moduli space are computed.  Section \ref{SigmaFGA} proves that the boundary values of the Thurston pullback map have a finite global attractor.  Section \ref{SigmaProps} presents further properties of this boundary map.  

Section \ref{sec:Twist} closes the study of $f$ with an outline of the solution to the twisting problem.  The pullback on curves is effectively used as an invariant for Thurston maps.%  A series of questions are posed in Section \ref{sec:Future}.

The following two paragraphs are fine, but a little off topic.

Bartholdi and Nekrashevych partially verify their twisted rabbit results using iterative properties of the moduli space map.  Their procedure is adapted here to correspondences on moduli space.  Suppose $f$ is a Thurston map and the commutative diagram in Section \ref{subsec:DefinitionsEndomorphisms} is produced.  The following determines the combinatorial class of $h\cdot f:=f\circ h$ by describing how to project the forward orbit of $\circledast\in\mathscr{T}_f$ under $\sigma_{h\cdot f}$ to moduli space  in an easily implemented way.  It is known that $\sigma_{h\cdot f}(\tau)=h\cdot\sigma_f(\tau)$ for $\tau\in\mathscr{T}_f$.  Thus $\sigma_{h\cdot f}(\circledast)$ is represented by the path $\gamma_h$ (corresponding to $h$) in $\mathscr{M}_f$ based at $z_0$.  Higher iterates are found inductively.  Let $\ell_h$ be the path representing $\sigma_{h\cdot f}^{\circ n}(\circledast)$: the path representing $\sigma_{h\cdot f}^{\circ n+1}(\circledast)$ is $\gamma_h\cdot X(Y^{-1}(\ell_h)[A(\circledast)])$ where by convention $\gamma_h$ is traversed first in the positive direction.  Thus in theory, one could find where the endpoint of the path converges to as $n$ gets large.  

Having shown how to compute the fate of this endpoint under iteration, the results must now be interpreted.  The orbit of $\circledast \in \mathscr{T}_f$ under iteration of $\sigma_{h\cdot f}$ either converges to some point in Teichm\"uller space and $h\cdot f$ is equivalent to a rational map, or the orbit of $\circledast$ escapes to the boundary and $h\cdot f$ is obstructed.  Thus, if the path just constructed escapes to the boundary of moduli space, $h\cdot f$ is clearly obstructed.  On the other hand, the correspondence has at least one fixed point corresponding to each possible rational map with the critical portrait of $f$, and if the path converges to one of these points as $n\rightarrow\infty$, then $h\cdot f$ is combinatorially equivalent to the corresponding rational map.  This procedure is exhibited for two examples in Figure 6 of \cite{BN06}.  

This paragraph is OK too:

The map on moduli space for the rabbit polynomial $F$ is used to calculate the image of $\sigma_F$ using the following method.  Let $g$ be the map on moduli space with a fixed basepoint $z_0$ as in \cite{K}.

\centerline{\xymatrix{\mathscr{T}_F \ar[r]^-{\sigma_F} \ar[d]^{\pi}
&\mathscr{T}_F\setminus\{\pi^{-1}(-1)\}\ar[d]^{\pi}\\ \mathscr{M}_F & \mathscr{M}_F\setminus\{-1\}
\ar[l]^-{g}}}
\noindent Paths starting at $z_0$ in $\mathscr{M}_F$ can be lifted by $g$ to paths in $\mathscr{M}_F\setminus\{-1\}$ that start at $z_0$. Since the universal cover is defined to be the space of homotopy classes of paths in $\mathscr{M}_F$ that begin at $z_0$, one can fix an identification of $\mathscr{T}_F$ with $\mathbb{D}$ to generate the image of $\sigma_F:\mathbb{D}\longrightarrow\mathbb{D}$.  Another application of the map on moduli space comes from \cite{BN06} where Bartholdi and Nekrashevych use this map to determine the combinatorial equivalence class of the rabbit twisted by $T$ and $T^{-1}$.  This is done by drawing paths in moduli space that correspond to these twists, and examining longterm behavior of successive lifts of this path.

\end{comment}

\section{Thurston Maps with $n$ Postcritical Points}
\label{sec:npoints}

Though later sections will only require the case of four postcritical points, we present Thurston theory in the more general setting.  

\subsection{Notation, Definitions, and Examples}
\label{sec:notationanddefns}

The standard oriented two-sphere is denoted by $S^2$ and the
Riemann sphere by $\widehat{\mathbb{C}}$. Let $F:S^2\longrightarrow
S^2$ be a degree $d$ branched covering where $C_F$ is the set of critical points and $V_F$ is the set of critical values.  Denote by $\text{deg}(F,x)$ the local degree of $F$ at the point $x\in S^2$.  The \textit{postcritical set} of $F$
is given by $$P_F=\bigcup_{i>0}F^i(C_F).$$ If $|P_F|$ is finite, $F$ is said to be \textit{postcritically
finite}. A \textit{Thurston map} is a postcritically finite orientation-preserving branched cover
$F:S^2\longrightarrow S^2$ where $\text{deg}(F)\geq 2$.  Define an equivalence relation on the set of Thurston maps as follows:

\begin{definition}\label{ThurstonEquivalence}
Let $F$ and $G$ be Thurston maps with postcritical sets $P_F$ and $P_G$ respectively.  Then $F$ is \textit{Thurston equivalent} to $G$ if there are orientation preserving homeomorphisms 
$h_0,h_1:(S^2,P_F)\longrightarrow(S^2,P_G)$ with $h_0$ homotopic to $h_1$ rel $P_F$, so that the following commutes:
\end{definition}

\centerline{ \xymatrix{(S^2,P_F) \ar[r]^{h_1} \ar[d]_F &{(S^2,P_G)} \ar[d]^{G} \\
(S^2,P_F) \ar[r]^{h_0} &{(S^2,P_G)}}} 

A simple closed curve $\gamma$ in $S^2\setminus P_F$ is \textit{essential} if each component of
$S^2\setminus{\gamma}$ intersects $P_F$ in at least two points.  It is called \textit{peripheral} if it bounds a disk containing only a single point of $P_F$.  A \textit{multicurve} is a collection
$\Gamma=\{\gamma_1,...\gamma_k\}$ of disjoint essential, simple closed curves where the elements of the
collection are pairwise non-homotopic. Use $\mathscr{C}_F$ to denote the set of homotopy classes of essential simple closed curves in $S^2\setminus P_F$.  Denote by
$\mathbb{R}[\mathscr{C}_F]$ the free $\mathbb{R}$-module over $\mathscr{C}_F$. The Thurston
linear map $\lambda_F:\mathbb{R}[\mathscr{C}_F]\longrightarrow\mathbb{R}[\mathscr{C}_F]$ is defined by
$$\lambda_F(\gamma)=\sum_{\gamma'}\sum_{\gamma'\simeq\delta\subset{F^{-1}(\gamma)}}
\frac{1}{\text{deg}(F:\delta\rightarrow\gamma)}\cdot \gamma'$$ where $\gamma$ and $\gamma'$ are single
components of multicurves and the outer sum is over all $\gamma'$ homotopic to preimages of $\gamma$.  A \textit{Thurston obstruction} is a nonempty multicurve $\Gamma$ so that $\mathbb{R}[\Gamma]$ is invariant under $\lambda_F$, and the spectral radius of $\lambda_F$ is
greater than or equal to $1$.

We now state Thurston's
theorem, referring to \cite{DH93} for a proof and \cite{Mi06} for a treatment of Latt\`es maps.

\begin{theorem} \label{thm:Thurston}
 Let $F$ be a Thurston map not equivalent to a Latt\`es map. Then F is Thurston equivalent to a rational function if and only if
there are no obstructions. If this rational function exists, it is unique up to M\"obius conjugation.
\end{theorem}

A crucial tool in the proof of Thurston's theorem is the Thurston pullback map on Teichm\"uller space. It can be shown that a Thurston map is equivalent to a rational function if and only
if its associated pullback map has a fixed point.

\begin{definition} The Teichm\"uller space for a Thurston map $F$ is defined to be
$$\mathscr{T}_F=\{\phi:(S^2,P_F)\longrightarrow \widehat{\mathbb{C}}\}/\sim$$ where $\phi_1\sim\phi_2$ if
and only if there is a M\"obius transformation $M$ so that $\phi_2$ is isotopic to $M\circ\phi_1$ rel $P_F$.
\end{definition}

Define the \textit{moduli space for $F$} to be $$\mathscr{M}_F:=\{\iota:P_F\hookrightarrow
\widehat{\mathbb{C}}\}/\approx$$ where each $\iota$ is an injection and $\iota_1\approx\iota_2$ if there is a M\"obius transformation $M$ so
that $M\circ\iota_1=\iota_2$. There is an obvious projection
$\pi_F:\mathscr{T}_F\longrightarrow\mathscr{M}_F$ defined by $\pi([\phi])=[\phi|_{P_F}]$ and it is known to be a universal cover mapping.   Denote by
$\text{Homeo}(S^2)$ the group of orientation-preserving homeomorphisms from $S^2$ to itself where from now on, by convention, all homeomorphisms are orientation preserving. Denote by
$\text{Homeo}(S^2,P_F)$ the subgroup of $\text{Homeo}(S^2)$ that fixes $P_F$ pointwise, and let
$\text{Homeo}_0(S^2,P_F)$ be the path component of the identity in $\text{Homeo}(S^2,P_F)$. The
\textit{pure mapping class group} of $S^2$ with respect to $P_F$ is
$$\text{PMCG}(S^2,P_F)=\text{Homeo}(S^2,P_F)/\text{Homeo}_0(S^2,P_F)$$ where $\text{Homeo}_0(S^2,P_F)$ acts on the right by post-composition.  The deck group of the universal cover $\pi_F$ is the
pure mapping class group acting by precomposition on representing homeomorphisms of points in $\mathscr{T}_F$.

To define the Thurston pullback map for $F$, choose a representative $\phi$ of
$\tau\in\mathscr{T}_F$. Pull back the complex structure on $\widehat{\mathbb{C}}$
by the Thurston map $F\circ\phi:(S^2,P_F)\rightarrow(\widehat{\mathbb{C}},\phi(P_F))$, and use the
uniformization theorem to conclude that $S^2$ with this new complex structure is isomorphic to
$\widehat{\mathbb{C}}$ by some $\tilde{\phi}$, unique up to postcomposition by
$\text{Aut}(\widehat{\mathbb{C}})$. We then have the following commutative diagram where $F_\tau$ is
defined to be the holomorphic composition $\phi\circ F\circ\tilde{\phi}^{-1}$ which must be a rational map of degree $\text{deg}F$.

\centerline{ \xymatrix{(S^2,P_F) \ar[r]^{\tilde{\phi}} \ar[d]_F
&{(\widehat{\mathbb{C}},\tilde{\phi}(P_F))} \ar[d]^{F_\tau} \\ (S^2,P_F) \ar[r]^{\phi}
&{(\widehat{\mathbb{C}},\phi(P_F))}}}

\begin{definition} The \textit{Thurston pullback map} $\sigma_F:\mathscr{T}_F\longrightarrow\mathscr{T}_F$
is defined by $\sigma_F(\tau)=[\tilde{\phi}]$. \end{definition}

An important tool that has emerged over the past decade for studying Thurston theory has been the use of correspondences on moduli space \cite{Pil10,BEKP,KPS}. In many situations, the right-hand map in the correspondence is an inclusion, and the correspondence can simply be thought of as a map on moduli space.  See \cite{K} for explicitly computed examples of maps on moduli along with a discussion of the general theory of correspondences on moduli space.

In general, one should not expect the correspondence on moduli space to be a map on moduli space.  This can be seen from the correspondence for $f(z)=\frac{3z^2}{2z^3+1}$ which will be presented here following the discussion of \cite{BEKP}.  Let $\omega=-\frac{1}{2}+\frac{\sqrt{3}}{2}i$. Denote by $F_0=\frac{P_0}{Q_0}$ and $F_\infty=\frac{P_\infty}{Q_\infty}$ two degree $3$ rational functions with three simple critical points at $1,\omega,\bar{\omega}$, where \[P_0(z)=3z^2, Q_0(z)=2z^3+1, P_\infty(z)=z^3+2, Q_\infty(z)=3z.\]
Direct computation shows that both $F_0$ and $F_\infty$ have the same mapping properties as $f$ on the three simple critical points $1,\omega,\bar{\omega}$, and that the fourth simple critical points of $F_0$ and $F_\infty$ are $0$ and $\infty$ respectively.  Let $F=\frac{P}{Q}$ be a degree three rational map with simple critical points $1,\omega,\bar{\omega}$ on which $F$ has the same mapping properties as $f$.  Since the numerators of $F-F_0=\frac{PQ_0-QP_0}{QQ_0}$ and
$F-F_\infty=\frac{PQ_\infty-QP_\infty}{QQ_\infty}$ are both scalar multiples of $(z^3-1)^2$, there exists
$\alpha=[a,b]\in\mathbb{P}^1$ so that $a\cdot(PQ_\infty-QP_\infty)+b\cdot(PQ_0-QP_0)=0.$ One can solve this equation
for $\frac{P}{Q}$ and add a subscript to yield $$F_{\alpha}=\frac{aP_\infty+bP_0}{aQ_\infty+bQ_0}.$$

\noindent Finding the fourth critical point and critical value of $F_{\alpha}$ leads to the following
commutative diagram where $$X(z)=z^2,Y(z)=\frac{z(z^3+2)}{2z^3+1}$$ and
$$y=\pi(\tau), x=\pi\circ\sigma_f(\tau), A(\tau)=\frac{x^2-y}{2xy-2}.$$  Use $\Theta$ to denote the set of cube roots of unity, and $\Theta'$ is the set of
sixth roots of unity. \[\xymatrix{ & \mathscr{T}_f \ar[dd]_{\pi}\ar[rr]^{\sigma_f} \ar[dr]^A & &
\mathscr{T}_f \ar[dd]^{\pi} \\ &&\rs-\Theta' \ar[dl]_Y \ar[dr]^X &\\ & \rs-\Theta & & \rs-\Theta.} \]

%Maybe mention topologies somewhere?

%I'm not sure but maybe "fibration" is "Serre fibration" or "locally trivial fibration"? Maybe you can recall the definition or give a reference?
%Somewhere in the proof of injectivity of \Phibar you use this as well.

%================================================= 
\subsection{Virtual Endomorphisms}
\label{subsec:DefinitionsEndomorphisms}

Let $H$ be a finite index subgroup of a group $G$.  A virtual endomorphism is a homomorphism $\phi:H\rightarrow G$.  The purpose of this section is to describe the relationship between two important virtual endomorphisms, one defined on the fundamental group of moduli space, and the other on the pure mapping class group of the dynamical plane.  This relationship was used in \cite{BN06}, and is crucial here as well.

As before, let $F$ be a Thurston map of degree $d$ with $|P_F|>3$.  Suppose that the following diagram commutes, where $\circledast$ is the unique fixed point of $\sigma_F$, $z_0:=\pi(\circledast)$, $w_0:=A(\circledast)$, and the maps $X,Y,$ and $A$ are defined according to the discussion in \cite{K}:
\begin{align*}\label{Wspace}
\xymatrix{ & (\mathscr{T}_F,\circledast) \ar[dd]_{\pi}\ar[rr]^{\sigma_F} \ar[dr]^A & &
(\mathscr{T}_F,\circledast) \ar[dd]^{\pi} \\ && (\mathscr{W}_F,w_0) \ar[dl]_Y \ar[dr]^X &\\ & (\mathscr{M}_F,z_0) & & (\mathscr{M}_F,z_0).}
\end{align*}
%Fix $id:P_F\hookrightarrow S^2$ to be the basepoint of $\mathscr{M}_F$.       
The first virtual endomorphism is defined on a subgroup of  $\pi_1(\mathscr{M}_F,z_0)$, where the domain of definition is
\[H=\{[\gamma]\in\pi_1(\mathscr{M}_F,z_0)|\, \gamma \text{ lifts to a loop } \tilde{\gamma} \text{ based at } w_0 \text{ under } Y\}.\]  
It is clear that this subgroup has finite index since $Y$ restricts to a cover.  Define the virtual endomorphism $\phi_F:H\rightarrow\pi_1(\mathscr{M}_F,z_0)$ by $\phi_F([\gamma])=X_*([\tilde{\gamma}])$, where $X_*$ is the induced map on the fundamental group.

The second virtual endomorphism is defined on a subgroup of the pure mapping class group of the dynamical plane.  The domain of definition is $H_F=\{[h]\in\mathrm{PMCG}(S^2,P_F)|\text{ there exists } \tilde{h} \text{ as follows }\}$ where $\tilde{h}$ is a homeomorphism that fixes the points in $P_F$ and makes the following commute:

\centerline{ \xymatrix{(S^2,P_F) \ar[r]^{\tilde{h}} \ar[d]_F &{(S^2,P_F)} \ar[d]^{F} \\ (S^2,P_F) \ar[r]^{h}
&{(S^2,P_F)}}}

\noindent The virtual endomorphism $\psi:H_F\rightarrow\mathrm{PMCG}(S^2,P_F)$ is defined by $\psi(h)=\tilde{h}$.  

The relationship between these two virtual endomorphisms is completely understood and was used in \cite{BN06}.  We briefly describe an explicit bijection between the points in $\mathscr{T}_F$ and homotopy classes of paths starting at the identity map in moduli space where homotopies fix endpoints of paths for all time \cite[Thm 2.4]{Lod}.  Given a point in Teichmuller space, choose a normalized representive $\phi$ that fixes three points in $P_F$.  It is well known that $\phi$ is isotopic to the identity fixing the three points for all time \cite{Fa}, and the trace of $P_F$ yields a path in moduli space based at the identity.  Conversely, given a path in moduli space based at the identity, isotopy extension yields a path in $\mathscr{T}_F$ based at the identity, and the point corresponding to the path in moduli space is chosen to be the endpoint of the path in $\mathscr{T}_F$.  This explicit bijection restricts to a bijection between $H$ and $H_F$ which has been shown to be natural \cite[Thm 2.6]{Lod}.

\subsection{Schreier Graphs and the Reidemeister-Schreier Algorithm}
\label{subsec:ReidemeisterSchreier}

The Reidemeister-Schreier algorithm will be used to compute the virtual endomorphism on the fundamental group.  The discussion of the algorithm given here follows the more abstract presentation given in \cite{Bog}.  Let $F:S^2\rightarrow S^2$ be a degree $d$ connected branched cover with critical values $V_F$.  Fix a basepoint $z_0\in S^2\setminus F^{-1}(V_F)$ with image $w_0=F(z_0)$.  Note that the restriction of $F$ to the complement of $F^{-1}(V_F)$ is a cover without branching, which by abuse of notation is also called $F$.  It is evident that $G:=\pi_1(S^2\setminus V_F,w_0)$ is a finitely generated free group on $|V_F|-1$ generators, and the basis set $S$ for $G$ is chosen to be a bouquet of $|V_F|-1$ disjoint oriented loops based at $w_0$, where each loop bounds a unique point of $V_F$.  Let $F_*$ be the induced map on the fundamental group and observe that since $G$ is free, the subgroup $H:=F_*(\pi_1(S^2\setminus F^{-1}(V_F),z_0))$ corresponding to the cover is free.  Schreier's formula \cite[p.66]{Bog} can be used to compute the rank of $H$:
\[\text{rank}(H)=d\cdot(|S|-1)+1.\]

The \textit{Schreier graph} is a labeled directed graph that can be used to find explicit generators for $H$.  This graph can be produced abstractly, but for the sake of this discussion it is embedded in the restricted domain of $F$.  The vertices of the Schreier graph are the points in $F^{-1}(z_0)$. If $v_0$ is a vertex and $s\in S$, denote the unique lift of $s$ starting at $v_0$ by $F^{-1}(s)[v_0]$.  A directed edge $e$ connects $v_0$ to $v_1$ if there is an element $s\in S$ so that the endpoint of $F^{-1}(s)[v_0]$ is $v_1$.  The vertex $v_0$ is called the \textit{initial vertex} of edge $e$ and is denoted by $i(e)$.  The vertex $v_1$ is called the \textit{terminal vertex} of $e$ and is denoted by $t(e)$.  If $e$ is an edge in the Schreier graph, the \textit{label} of $e$ is defined by $\ell(e)=s$ where $s$ is the unique element so that the endpoint of $F^{-1}(s)[i(e)]$ is $t(e)$.  Denote the reverse of the edge $e$ by the formal symbol $e^{-1}$, and extend the notion of initial and terminal vertex in the obvious way: $i(e^{-1}):=t(e)$ and $t(e^{-1}):=i(e)$.  Define a finite path in the Schreier graph to be a finite list of edges or reversed edges $e_1e_2...e_n$, where for all $j$, $t(e_j)=i(e_{j+1})$, $1\leq j <n$.  The notion of label can easily be extended inductively to finite paths:
\begin{itemize}
\item $\ell(\text{empty path}):=1$
\item for any edge $e$, $\ell(e^{-1}):=\ell(e)^{-1}$
\item for paths $p,q$ in the Schreier graph with $t(p)=i(q)$, $\ell(pq):=\ell(p)\ell(q)$.  
\end{itemize}
Choose a maximal subtree $\Delta$ of the Schreier graph, and for some vertex $v$ in the Schreier graph, denote by $p_{v}$ the unique shortest path from $z_0$ to $v$ contained in $\Delta$.  Then for every edge $e$ where neither $e$ nor $e^{-1}$ is contained in $\Delta$, define $p_e$ to be the loop $p_{i(e)}ep_{t(e)}$ in the Schreier graph based at $z_0$.  A classical theorem \cite[p.67]{Bog} asserts that $H$ is generated by:
\[\{\ell(p_e)|e \text{ is an edge where } e \text{ and } e^{-1} \text{ are not contained in }\Delta\}.\]

Define the \textit{Schreier transversal} of a Schreier graph with fixed choice of maximal subtree $\Delta$ to be the following set of $d$ elements: 
\[T=\{\ell(p_{v}^{-1})|v \text{ is a vertex in the Schreier graph}\}.\] Having found generators for $H$, we now describe the Reidemeister-Schreier algorithm which accepts as input a word $h\in H$ written in terms of $S\cup S^{-1}$, and rewrites it as a product of the basis elements for $H$ (and their inverses).  First some notation is defined: let $p$ be the path in the Schreier graph starting at $z_0$ with $\ell(p)=g$ for some $g\in G$, and let $t$ be the unique shortest path in $T$ starting at $z_0$ and having the same endpoint as $p$.  Define $\overline{g}$ to be $\ell(t)$.  
Then for $t\in T$ and $s\in S$, define $\gamma(t,s):=ts(\overline{ts})^{-1}\in H$.  Then given $h=s_1s_2...s_k,$ $s_j\in  S\cup S^{-1}$, the rewriting of $h$ is given by
\[h=\gamma(1,s_1)\gamma(\overline{s_1},s_2)...\gamma(\overline{s_1s_2...s_{k-2}},s_{k-1})\gamma(\overline{s_1s_2...s_{k-1}},s_k).\]

\subsection{Iterated Monodromy Groups and Wreath Recursions}
\label{sec:IMG}
The algebraic tools used to solve the twisted rabbit problem are presented here.  These tools apply to the more general setting of partial self-coverings \cite{BN06}, but we simply take $F:S^2\rightarrow S^2$ to be a finite branched cover.

First, fix a basepoint $z_0\in S^2\setminus P_F$ and a positive integer $k$.  Denote the unique lift of $\gamma$ under $F^k$ starting at $z\in F^{-k}(z_0)$ by $F^{-k}(\gamma)[z]$.   The group $\pi_1(S^2\setminus P_F,z_0)$ acts on the right of the set $F^{-k}(z_0)$ by declaring $z\cdot\gamma$ to be the endpoint of $F^{-k}(\gamma)[z]$.  The induced action on the disjoint union $\coprod_{n\geq 0} F^{-n}(z_0)$ is called the \textit{iterated monodromy action}. We now construct the tree of preimages and describe the iterated monodromy action on this tree.  Let $z_0$ be the root of the tree and each element of $\coprod_{n>0} F^{-n}(z_0)$ a vertex.  Join each vertex $z\in F^{-n}(z_0)$ to each of the vertices in $F^{-1}(z)\subset F^{-n-1}(z_0)$.  It is easily shown that the iterated monodromy action (extended to the edges in the obvious way) acts by automorphisms of this tree, but this action is not necessarily faithful.  Thus we define the \textit{iterated monodromy group} IMG($F$) as follows:
\[\text{IMG}(F)=\pi_1(S^2\setminus P_F,z_0)/ker\]
where $ker$ is the kernel of the iterated monodromy action.

There is a convenient way to index these trees along with their automorphisms.  Let $X=\{1,...,d\}$ and denote the set of strings of length $n$ in the letters $1,...,d$ by $X^n$, and the set of infinite strings in these $d$ letters by $X^{*}$.  Identify the tree of preimages of $z_0$ with the set $X^{*}$ as follows:
\begin{itemize}
\item Identify $z_0$ with the empty word
\item Choose a bijection between $z\in F^{-1}(z_0)$ and $x\in X$. For each $x\in X$, choose paths $\ell_x$ between $z_0$ and the point in $F^{-1}(z_0)$ corresponding to $x\in X$.
\item Suppose $v$ is a word in the alphabet $X$ and $x\in X$.  The identification here will be made using induction.  Let $z\in F^{-n}(z_0)$ corresponds to $v$; identify $vx$ with the endpoint of the path $F^{-n}(\ell_x)[z]$.
\end{itemize}
The iterated monodromy action can be conjugated by this bijection to yield an action of $\pi_1(S^2\setminus P_F,z_0)$ on $X^*$.  Section \ref{sec:knownresults} describes how the iterated monodromy group solved the twisted rabbit problem.

Another important tool in \cite{BN06} for understanding the action of the fundamental group on $X^*$ are wreath recursions.  Denote by $S_d$ the symmetric group on $d$ letters.  Multiply elements of $S_d$ in the following way:
\[(\,1\quad 4\quad2\,)(\,1\quad 3\quad 4\,)=(\,2\quad 3\quad 4\,).\]
Define the \textit{wreath product} $G\wr S_d$ for some group $G$ to be $G^d\rtimes S_d$ where $S^d$ acts on the $d$-fold product $G^d$ by permutation of coordinates.  Thus, if $\langle\langle g_1,...,g_d\rangle\rangle\sigma$ and $\langle\langle h_1,...,h_d\rangle\rangle\tau$ are elements of $G\wr S_d$, multiplication is defined by:
\[\langle\langle g_1,...,g_d\rangle\rangle\sigma\langle\langle h_1,...,h_d\rangle\rangle\tau=\langle\langle g_1 h_{\sigma(1)},...,g_d h_{\sigma(d)}\rangle\rangle\sigma\tau\]
For example, if $G$ is the free group on generators $\alpha$ and $\beta$,
\[\langle\langle 1,\beta\alpha,\alpha,\beta^{-1}\rangle\rangle(\,1\quad 4\quad 2\,)\langle\langle\beta\alpha,\alpha^{-1},1,\beta\rangle\rangle(\,1\quad 3\quad 4\,)\]
\[=\langle\langle\beta,\beta\alpha\beta\alpha,\alpha,\beta^{-1}\alpha^{-1}\rangle\rangle(\,2\quad 3\quad 4\,)\]
A \textit{wreath recursion} is a homomorphism $\Phi:G\rightarrow G\wr S_d$.  For $x\in X$, denote the restriction to the $x$th coordinate of $\Phi(g)$ by $g|_x$.  For $v\in X^*$ and $x\in X$, inductively define $g|_{xv}:=(g|_x)|_v$.  Define the action of $G$ on $X$ by projecting to the second factor $G\wr S_d\rightarrow S_d$, and observe that this action extends to the associated action of $G$ on $X^*$ by the formula $(vx)^g=v^g x^{g|v}$.  The following proposition coming from \cite[p.7]{BN06} where by convention, path multiplication is defined by following the leftmost path in the positive direction and ending with the rightmost.  The bar denotes the reverse path.

\begin{theorem} The action of $\pi_1(S^2\setminus P_F, z_0)$ on $X^*$ is the action associated with $\Phi:\pi_1(S^2\setminus P_F,z_0)\rightarrow \pi_1(S^2\setminus P_F,z_0)\wr S_d$ given by
\[\Phi(\gamma)=\langle\langle\ell_1\gamma_1\overline{\ell}_{k_1},\ell_2\gamma_2\overline{\ell}_{k_2},...\ell_d\gamma_d\overline{\ell}_{k_d}\rangle\rangle\rho\]
where $\gamma_i=F^{-1}(\gamma)[z_i]$, $z_i$ is the endpoint of $\ell_i$, $k_i$ is the element of $X$ corresponding to $z_i$, and $\rho$ is the permutation defined by $i\mapsto k_i$ for all $i\in X$.
\end{theorem}

For a given wreath recursion, there is a common way to define $|X|$ corresponding virtual endomorphisms.  For $i\in X$, the domain of the virtual endomorphism $\phi_i:H_i\rightarrow G$ is the subgroup $H_i<G$ where $h\in H_i$ if the permutation factor of $\Phi(h)$ fixes $i$.  Then $\phi_i(h)$ is defined to be the projection to the $i$th component of $\Phi(h)$.

The contracting properties of wreath recursions are important for reducing infinite problems to finite ones.  A wreath recursion $\Phi:G\rightarrow G\wr S_d$ is contracting if there is a finite $\mathcal{N}\subset G$ so that for every $g\in G$, there is a positive number $n_0$ so that $g|_v\in\mathcal{N}$ for all words $v$ of length greater than $n_0$.  The smallest such $\mathcal{N}$ is called the nucleus of the action.  There is a useful characterization of the contracting property in \cite[p.57]{Nek}: \textit{ A wreath recursion defined on a group $G$ with generating set $S=S^{-1}$ and $1\in S$ is contracting if and only if there exists a finite set $\mathcal{N}$ and a number $k\in\mathbb{N}$ so that}
\[((S\cup\mathcal{N})^2)|_{X^k}\subset\mathcal{N}\]
There is also a notion of contraction for virtual endomorphisms.  Let $\phi:\text{dom}\phi\rightarrow G$ be a virtual endomorphism.  The spectral radius $\rho_\phi$ of $\phi$ is
\[\rho_\phi=\limsup_{n\rightarrow\infty}\sqrt[n]{\limsup_{g\in\text{dom}\phi^n,|g|\rightarrow\infty}\frac{|\phi^n(g)|}{|g|}}\]
where $|\cdot|$ denotes word length with respect to some fixed generating set of $G$.  In \cite[p.9]{BN06} we find the following proposition:

\begin{proposition}
Let $\Phi:G\rightarrow G\wr S_d$ be a wreath recursion, and let $\phi$ be an associated virtual endomorphism.  If $\Phi$ is contracting, then $\rho_\phi<1$.  If the action of $G$ is transitive on every level $X^n$ and $\rho_\phi<1$, then the wreath recursion $\Phi$ is contracting.
\end{proposition}

\subsection{Summary of Known Results}
\label{sec:knownresults}

The primary goals of this paper is to solve a twisting problem for a rational function, and compute the pullback on curves for this same function.  Earlier progress related to these problems will be discussed here, beginning with the solution to the twisted rabbit problem in \cite{BN06}.  More general work on twisted polynomial problems can be found in \cite{BN08,Nek09,Ke11}.

A Thurston map $F$ is a \textit{topological polynomial} if there is some point $\infty\in S^2$ so that $F^{-1}(\infty)=\{\infty\}$.  A \textit{Levy cycle} is a multicurve $\{\gamma_0,\gamma_1,...,\gamma_{n-1}\}$ so that the single nonperipheral component of the preimage of each $\gamma_i$ under $F$ is homotopic to $\gamma_{i-1}$, and $F$ maps each $\gamma_{i-1}$ to $\gamma_{i}$ by degree $1$ for each $i$ (indices are considered mod $n$).  For a topological polynomial $F$, it has been shown that every obstruction contains a Levy cycle.

Let $f$ be a quadratic polynomial with three finite post-critical points that are cyclically permuted.  All such polynomials can be found up to conjugacy by letting $f_c(z)=z^2+c$ and solving the equation $f_c^{\circ 3}(0)=0$ for $c$ where the root $c=0$ is ignored.  This leaves $c=-1.7549,-0.1226+0.7449i, -0.1226+0.7449i$ which correspond to the airplane polynomial $f_A$, rabbit polynomial $f_R$, and corabbit polynomial $f_C$ respectively.  Postcompose $f_R$ by a Dehn twist along a curve that avoids the postcritical set.  Obstructions for this twisted rabbit must contain Levy cycles, but such a cycle cannot exist for topological reasons.  Thus, the twisted rabbit is equivalent to one of $f_A$, $f_R$, or $f_C$.  Crucial to the work of Bartholdi and Nekrashevych is their Corollary 3.3, which asserts that two Thurston equivalent quadratic topological polynomials with identical postcritical set will have identical iterated monodromy groups.  After identifying the marked spheres for $f_A$, $f_R$, and $f_C$, the iterated monodromy groups of all three polynomials are shown to have different nuclei.  Thus, if one can compute the nucleus of the iterated monodromy group of a twisted rabbit, one can determine its Thurston class.

Bartholdi and Nekrashevych fix two generators of the pure mapping class group of the plane for $f_R$ which they denote $S$ and $T$.  Denote by $\psi:H_{f_R}<\text{PMCG}(\rs,P_{f_R})\rightarrow \text{PMCG}(\rs,P_{f_R})$ the virtual endomorphism on the pure mapping class group.  This virtual endomorphism is extended to a function $\overline{\psi}$ on all of $\text{PMCG}(\rs,P_f)$ so that $h\circ f_R$ is Thurston equivalent to $\overline{\psi}(h)\circ f_R$.  Thus one should try to understand the fate of $h$ under iterated application of $\overline{\psi}$.  In fact $\overline{\psi}$ is contracting \cite[Prop 4.2]{BN06}, and any $h$ lands in the set $\{id,T,T^{-1}\}$ after finitely many iterations.  If the orbit of $h$ lands on $id, T,$ or $T^{-1}$, it is shown that $f_R\circ h$ is equivalent to the rabbit, airplane, or corabbit respectively using iterated monodromy group computations \cite[Thm 4.8]{BN06}.

The next case of the twisted $z^2+i$ is dealt with in a similar way, but the twisted $z^2+i$ is equivalent to either $z^2+i, z^2-i$, or a $\mathbb{Z}$-parameter family of obstructed examples that are shown to be inequivalent to each other.  The last kind of quadratic with three finite postcritical points is the preperiod 2, period 1 case.  Each twisting is equivalent to one of the three quadratic polynomials realizing the critical portrait, and as in the rabbit case, it is impossible to have obstructed twistings.

We now move on to a digest of known results about iterated preimages of multicurves under a Thurston map.  Let $\Gamma$ and $\Gamma'$ be multicurves in $\rs\setminus P_F$ for some Thurston map $F$.  The multicurve $\Gamma$ pulls back to $\Gamma'$ if the essential nonperipheral components of $f^{-1}(\Gamma)$ form a multicurve homotopic to $\Gamma'$, and denote this by $\Gamma'\overset{f}{\rightarrow}\Gamma$.  In \cite{Pil10} it is shown that if the virtual endomorphism corresponding to a Thurston map is contracting, then there is a finite global attractor for the pullback function on multicurves.  Pilgrim also shows that there is a finite number of completely invariant multicurves when $F$ is a rational function.  There is an analytic method to prove existence of a finite global attractor by looking at contracting properties of the map on moduli space which can be applied to all quadratics with periodic finite critical point \cite[Cor 7.2]{Pil10}.  However, in the case of a mixture of attracting and repelling fixed points for the correspondence on moduli space (which is true of the correspondence for $f(z)=\frac{3z^2}{2z^3+1}$), these methods don't apply directly.    

There are few explicit computations of the dynamics of the pullback function.  It is proven in \cite{Pil10} that the pullback function of the rabbit polynomial has a finite global attractor which is a cycle of length 3.  The pullback function for $z^2\pm i$ is shown by computation to be eventually trivial, as is the pullback function for $z^2-0.2282\pm 1.1151i$.

\end{comment}

\section{General facts in the case $|P_f|=4$}
\label{sec:FourPostcriticalPoints}

We now specialize to the case when $f$ is a Thuston map with $P_f=\{z_0,z_1,z_2,z_3\}$. Multicurves in this context are just curves, and this section presents ways that these curves can be encoded using a variety of familiar objects such as points in the Weil-Petersson completion of $\mathscr{T}_f$, the extended rational numbers  $\overline{\mathbb{Q}}:=\mathbb{Q}\cup\{1/0\}$, and certain words in free groups on two generators.  Continued fractions are discussed as well, since computing the boundary values of Thurston's pullback map amounts to transforming even continued fraction expansions in some prescribed way.  The assignment of rational slopes to curves is a somewhat different notion, and is deferred until Section \ref{sec:slope} to avoid confusion with the assignment of rational numbers presented here.

Define the subset $\Theta:=\{z_1,z_2,z_3\}\subset P_f$.  The universal cover $\mathbb{H}\rightarrow\rs\setminus\{1,\omega,\overline{\omega}\}$ in Figure \ref{fig:cover} is defined by first taking the Riemann map sending the ideal triangle in $\mathbb{H}$ with vertices $(0,1,\infty)$ to the unit disk in $\mathbb{C}$ with vertices $(\overline{\omega},1,\omega)$.  This map is then extended by reflection to all of $\mathbb{H}$.
\begin{figure}[h]
\centering
\includegraphics[width=100mm]{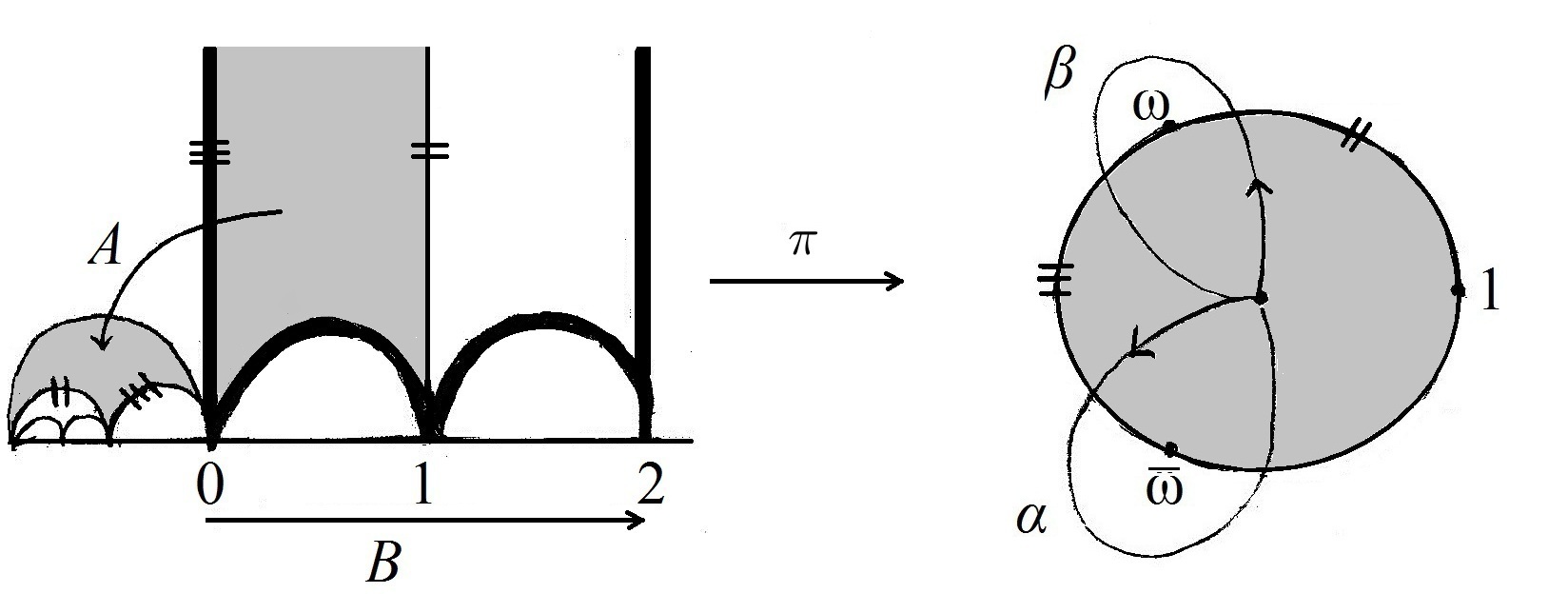}
\caption{Modular map used to define $\pi$}
\label{fig:cover}
\end{figure}
Postcompose this extended map by the unique M\"obius transformation sending $(1,\omega,\overline{\omega})$ to $(z_1,z_2,z_3)$ and denote the resulting map by $\pi:\mathbb{H}\rightarrow\rs\setminus\{z_1,z_2,z_3\}$.  The fundamental domain of the deck group of $\pi$ is chosen to be the half-open ideal quadrilateral with vertices $0,1,2$ and $\infty$ containing the arcs connecting $0$ to $1$ and $\infty$, but not containing the arcs connecting $2$ to $1$ and $\infty$.  The set $\pi^{-1}(z_0)$ intersects this fundamental domain at precisely one point $\tau_0$, and this is established to be the basepoint of the cover $\pi:(\mathbb{H},\tau_0)\rightarrow (\rs\setminus\{z_1,z_2,z_3\},z_0)$.

The deck group of $\pi$ is a well-known subgroup of the modular group, and we will fix a typical choice of generators of this group.  The following theory is classical and discussed in \cite{La}.  Denote by $\text{SL}_2(\mathbb{Z})$ the special linear group of $2\times 2$ matrices with integer coeffecients, and denote by $\text{PSL}_2(\mathbb{Z})$ the quotient of $\text{SL}_2(\mathbb{Z})$ by $\langle -I \rangle$ where $I$ denotes the identity matrix.  The group $\text{PSL}_2(\mathbb{Z})$ acts by orientation-preserving isometry on the upper halfplane by identifying $\left[\begin{smallmatrix} a&b\\ c&d \end{smallmatrix}\right]\in\text{PSL}_2(\mathbb{Z})$ with the M\"obius transformation $z\mapsto \frac{az+b}{cz+d}$.  The level two congruence subgroup $\Gamma(2)$ in $\text{SL}_2(\mathbb{Z})$ is defined to be the kernel of the obvious projection $\text{SL}_2(\mathbb{Z})\rightarrow \text{SL}_2(\mathbb{Z} / 2 \mathbb{Z})$, and the projectivized level two congruence subgroup is
$\text{P}\Gamma(2):=\Gamma(2)/\langle{-I}\rangle$.  Thus
\[\Gamma(2)=\{\left[ \begin{smallmatrix} a&b\\ c&d
\end{smallmatrix}\right]:a\equiv d\equiv 1, b\equiv c\equiv 0,
ad-bc=1\}.\]  
It is further known that $\text{P}\Gamma(2)$ is the deck group for the universal cover $\pi$, and  $\text{P}\Gamma(2)=\langle A,B\rangle$ where
$A = \left[
\begin{smallmatrix} 1&0\\ -2&1 \end{smallmatrix}\right]$, $B =\left[
\begin{smallmatrix} 1&2\\ 0&1 \end{smallmatrix}\right]$.

Since $A$ and $B$ generate the deck group, standard covering space theory gives a natural way to produce generators of $\pi_1(\rs\setminus\{z_1,z_2,z_3\},z_0)$ and $\text{PMCG}(\rs,P_f)$.  Choose an arc in the universal cover whose starting point is $\tau_0$ and whose endpoint is $A(\tau_0)$.  In like manner, choose oriented arcs connecting $\tau_0$ to $B(\tau_0)$ and $B^{-1}A^{-1}(\tau_0)$ (where $A^{-1}$ acts first).  Pushing these three arcs down by $\pi$ yield three loops $\alpha,\beta,$ and $\gamma$ respectively based at $z_0$.  Note that $\alpha$ bounds a singly-punctured disk and a doubly-punctured disk.  Orient $\alpha$ by declaring that the single puncture lies on the left side of $\alpha$, and orient $\beta$ and $\gamma$ in the same way.  Note that $\alpha,\beta,\gamma$ generate the fundamental group, and
\[\pi_1(\rs\setminus\{z_1,z_2,z_3\},z_0)=\langle\alpha,\beta\rangle=\langle\alpha,\beta,\gamma|\beta\alpha\gamma\rangle.\]

Now we will obtain generators for the mapping class group.  Suppose $\tau:[0,1]\rightarrow \rs\setminus\Theta$ is a path with $\tau(0)=z_0$.  Then the isotopy extension theorem
\cite{Hi76} guarantees the existence of an ambient isotopy
$$\phi:\rs\times I \longrightarrow \rs$$
where $\phi$ fixes small neighborhoods of $\Theta$ for
all $t$, and $\phi(\tau(t),t)=\tau(t)$.  The homeomorphism $\phi(\cdot,1)$ is called the \textit{ point push of $z_0$ along $\tau$}.  A point push along a positively-oriented simple loop can be written as a composition of a left and a right Dehn twist \cite[Fact 4.7]{Fa}, and in the setting here, the left Dehn twist is evidently trivial.  Thus pushing $z_0$ along the loops $\alpha,\beta,\gamma$ yield three right Dehn twists $T_\alpha, T_\beta, T_\gamma$ respectively that generate $\text{PMCG}(\rs,P_f)$ \cite{Bi74}.

When $|P_f|=4$, the model of Teichm\"uller space is taken to be $\mathbb{H}$, and the Weil-Petersson completion of Teichm\"uller space is $\mathbb{H}\cup\Qbar$ equipped with the horoball topology \cite{Wol09}.  Identify $\frac{p}{q}\in\overline{\mathbb{Q}}$ with the element $\left[
\begin{smallmatrix} p\\ q\end{smallmatrix}\right]$ of the projective line over the vector space $\mathbb{Z}^2$ denoted $P\mathbb{Z}^2$.  A left action of $\text{P}\Gamma(2)$ on $\overline{\mathbb{Q}}$ is induced from the action of $\Gamma(2)$ on $P\mathbb{Z}^2$ by left multiplication.  The set $\{\frac{0}{1}, \frac{1}{0}, -\frac{1}{1}\}$ is an orbit transversal for this action.  Note that $\frac{0}{1}, \frac{1}{0}$, and $-\frac{1}{1}$ are fixed by $A,B,$ and $B^{-1}A^{-1}$ respectively.  The stabilizer of some point $\frac{p}{q}$ is the set of group elements of the form $w^{-1}v^n w$ where $w\in$P$\Gamma(2)$ with $w.\frac{p}{q}\in\{\frac{0}{1}, \frac{1}{0}, -\frac{1}{1}\}$, $n$ an integer, and $v$ the unique element of $\{A,B,B^{-1}A^{-1}\}$ that fixes $w.\frac{p}{q}$.  The union $\bigcup_{\frac{p}{q}\in\overline{\mathbb{Q}}} \text{Stab}_{\text{P}\Gamma(2)}(\frac{p}{q})$ 
is defined to be the set of \textit{parabolic elements of} P$\Gamma(2)$.  Equivalently, an element $C\in\text{P}\Gamma(2)$ is parabolic if $(\text{trace}(C))^2=4$.  

Define the (unoriented) curve in $\rs\setminus P_f$ that corresponds to $\frac{p}{q}$ to be the core curve of the Dehn twist that comes from pushing the point $z_0$ in the positive direction along the loop that corresponds to $w^{-1}v w$.  For example, $\frac{9}{5}$ is fixed by left multiplication of the parabolic element $ BA^2(AB)A^{-2}B^{-1}\in \text{P}\Gamma(2)$, which corresponds to the Dehn twist that arises from pushing $z_0$ along the path $\beta^{-1}\alpha^{-2}\beta\alpha\alpha^2\beta$.  There are two pieces of data we associate to $\text{Stab}_{\text{P}\Gamma(2)}(\frac{p}{q})$ for each $\frac{p}{q}$: the element $*$ of the orbit transversal that lies in the orbit of $\frac{p}{q}$, and an essentially unique word $w$ so that $w.*=\frac{p}{q}$.  It is evident that $\text{Stab}_{\text{P}\Gamma(2)}(\frac{p}{q})=w\text{Stab}_{\text{P}\Gamma(2)}(*)w^{-1}$ where $\text{Stab}_{\text{P}\Gamma(2)}(*)$ is the infinite cyclic group generated by one of $A$,$B$, or $B^{-1}A^{-1}$ depending on $*$.

The following is an algorithm for finding an even continued fraction expansion of a reduced $\frac{p}{q}\in\mathbb{Q}$ closely related to the standard development in \cite{La}.  The input for the machine is a rational number $\frac{p}{q}$ and the output is $w$ and $*$ as above, along with the continued fraction notation for $\frac{p}{q}$ contained in the string $v$.  Let $p_0:=p$, $q_0:=q$ and let $k=0$, $w$ and $v$ the empty string.  Start at the beginning state of the machine in Figure \ref{fig:machine}.

Depending on which interval or singleton contains $\frac{p_k}{q_k}$, follow one of the seven outbound arrows to a new state.  If this new state has two concentric circles, append the expression inside the circles to the right of $v$, let $*:=\frac{p_k}{q_k}$, and terminate the algorithm. If the new state does not have two concentric circles, append the characters in the box lying on the arrow to the right of $v$.  Then use the following rules to compute the value of $\frac{p_{k+1}}{q_{k+1}}$ and the character to append to the right of the output string $w$:
\begin{itemize}
\item If $\frac{p_k}{q_k}\in(-\infty,-1)$, set $\frac{p_{k+1}}{q_{k+1}}:=B.\frac{p_k}{q_k}$, and append $B^{-1}$.
\item If $\frac{p_k}{q_k}\in(-1,0)$, set $\frac{p_{k+1}}{q_{k+1}}:=A^{-1}.\frac{p_k}{q_k}$, and append $A$.
\item If $\frac{p_k}{q_k}\in(0,1]$, set $\frac{p_{k+1}}{q_{k+1}}:=A.\frac{p_k}{q_k}$, and append $A^{-1}$.
\item If $\frac{p_k}{q_k}\in(1,\infty)$, set $\frac{p_{k+1}}{q_{k+1}}:=B^{-1}.\frac{p_k}{q_k}$, and append $B$.
\end{itemize}
Increment $k$ and repeat the process as described above in this paragraph until the process terminates.  The \textit{naive rational height} of $\frac{p}{q}$ is $\text{max}(|p|,|q|)$, and can be used to demonstrate that the algorithm terminates in finite time because the naive rational height of $\frac{p_k}{q_k}$ is always greater than that of $\frac{p_{k+1}}{q_{k+1}}$, except in the case when $\frac{p_k}{q_k}=\frac{1}{1}$ where the height is preserved.  The interpretation of the continued fraction notation $v=[a_0;a_1;...;a_n]$ just produced depends on the value of $a_n$.  If $a_n\neq\frac{0}{1},\frac{1}{0}$, we define
\[[a_0;a_1;a_2;...;a_n]:=a_0 + \cfrac{1}{a_1 + \cfrac{1}{a_2 + \cfrac{1}{ \ddots + \cfrac{1}{a_n} }}}.\]
If $a_n=\frac{0}{1}$, then $[a_0;a_1;a_2;...;a_n]:=[a_0;a_1;a_2;...;a_{n-2}]$.
Finally, if $a_n=\frac{1}{0}$,
$[a_0;a_1;a_2;...;a_n]:=[a_0;a_1;a_2;...;a_{n-1}].$
For example, $$\frac{7}{12}=A^{-1}.-\frac{7}{2}=A^{-1}B^{-1}.-\frac{3}{2}=A^{-1}B^{-1}B^{-1}.\frac{1}{2}=A^{-1}B^{-1}B^{-1}A^{-1}.\frac{1}{0},$$
and the machine yields the following continued fraction expansion:
\[\frac{7}{12}=[0;2;-2-2;2;\frac{1}{0}]=
  0+\cfrac{1}{
    2+\cfrac{1}{
       -4+\cfrac{1}{\frac{1}{2}}}}.\]

\begin{figure}[p]
\includegraphics[width=160mm, angle=270]{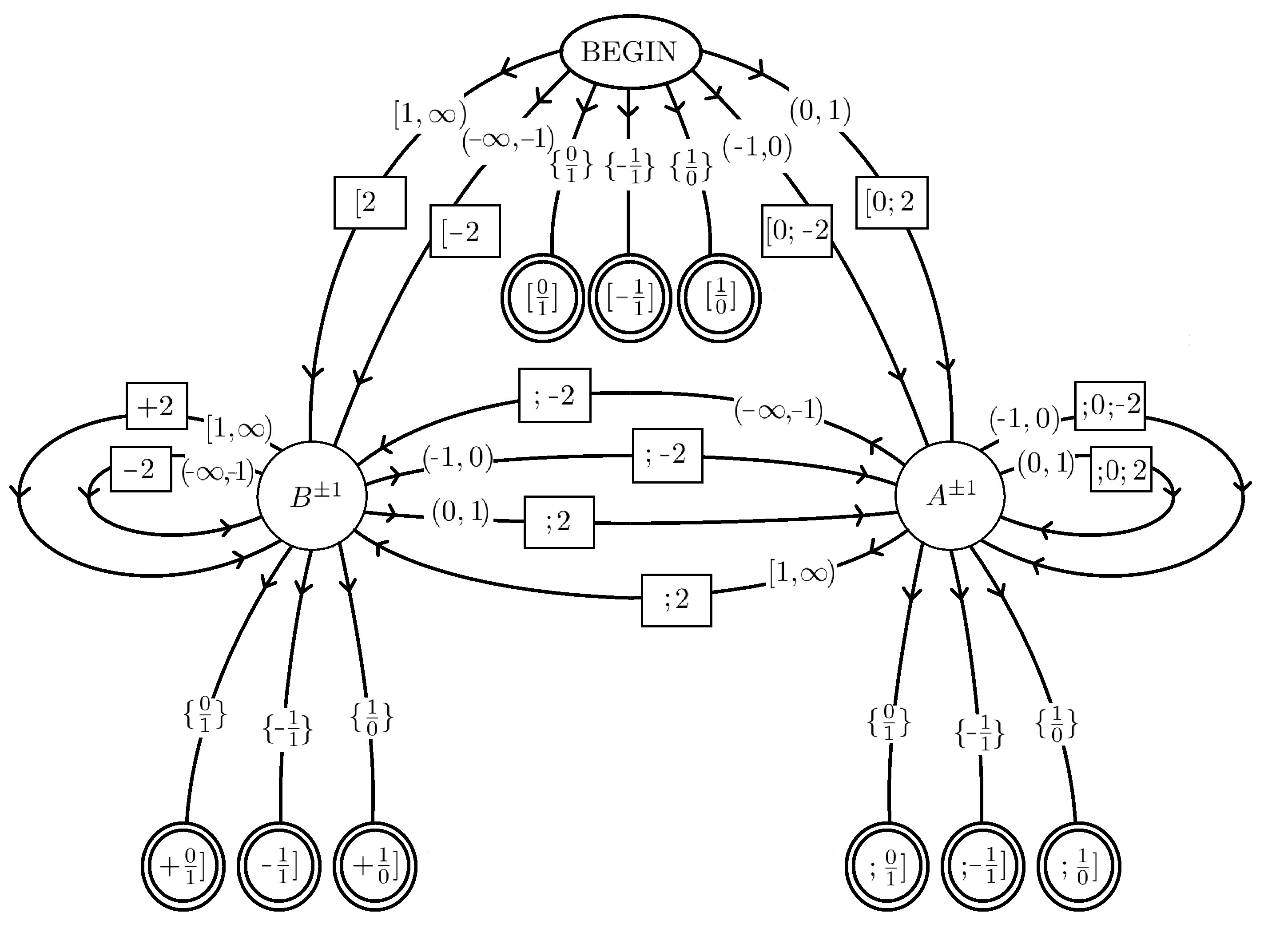}
\caption{Machine to compute even continued fraction expansion}
\label{fig:machine}
\end{figure}

For any fixed $*\in\overline{\mathbb{Q}}$ there are infinitely many reduced words $w\in\text{P}\Gamma(2)$ so that $w.*=\frac{p}{q}$, e.g. $BA^n.\frac{0}{1} =\frac{2}{1}$ for all integers $n$.  The algorithm presented above produces a word $w$ with the property that any other $w'$ with $w'.*=\frac{p}{q}$ contains $w$ as a subword (ignoring the trivial differences that arise from the fact that $A.\frac{1}{1}=B^{-1}.\frac{1}{1}$).  For example, suppose that $\frac{p}{q}=w'.\frac{0}{1}$ and that the algorithm produces $\frac{p}{q}=w.\frac{0}{1}$.  Then $(w')^{-1}w.\frac{0}{1}=\frac{0}{1}$ which means that $(w')^{-1}w$ lies in the maximal parabolic subgroup which fixes $\frac{0}{1}$.  Thus $(w')^{-1}w=A^n$ for some $n\in\mathbb{Z}$, and an examination of the algorithm from above allows one to conclude that $w$ must be a subword of $w'$ since the algorithm terminates as soon as $\frac{p_k}{q_k}=\frac{0}{1}$.

This section concludes with a discussion of the actions of $\text{PMCG}(\rs,P_f)$ and $\pi_1(\rs\setminus\Theta,z_0)$ on the set of extended rational numbers.  The action of $\text{P}\Gamma(2)$ on $\Qbar$ was defined to be a left action.  One could naturally define $\text{PMCG}(\rs,P_f)$ to act on the left, but we make it a right action because point pushing naturally identifies it with the fundamental group, which naturally acts on the right.  Thus, define $f\cdot g$ to mean $g\circ f$.  The isomorphism from $\pi_1(\rs\setminus\Theta,z_0)$ to $\text{PMCG}(\rs,P_f)$ is defined on generators as follows:
\[\alpha \mapsto T_\alpha\] 
\[\beta \mapsto T_\beta.\]
Let $G$ and $H$ be groups.  An \textit{antihomomorphism} from $G$ to $H$ is a function $\phi:G\rightarrow H$ so that $\phi(g_1 g_2)=\phi(g_2)\phi(g_1)$ for all $g_1,g_2\in G$.  The \textit{opposite} of a homomorphism $\phi:G\rightarrow H$ is the function $\tilde{\phi}:G\rightarrow H$ defined by $\tilde{\phi}(g_1 g_2)=\phi(g_2)\phi(g_1)$.  This is evidently an antihomomorphism.  The composition of two antihomomorphisms is a homomorphism.  The anti-isomorphism between $\text{P}\Gamma(2)$ and $\pi_1(\rs\setminus\Theta,z_0)$ is defined by postcomposing the homomorphism defined below by the opposite of the identity homomorphism:
\[A\mapsto \alpha\] 
\[B\mapsto \beta.\]

Thus, the deck group $\text{P}\Gamma(2)$ acts on the left, and the fundamental group and pure mapping class group act on the right.  The isomorphisms and anti-isomorphism described above yield isomorphic group actions on $\overline{\mathbb{Q}}$ in all three cases.  For example, $B^{-1}A^2BA^{-1}\in\text{P}\Gamma(2)$, $T_{\alpha}^{-1}\cdot T_{\beta}\cdot T_{\alpha}^2\cdot T_{\beta}^{-1}\in\text{PMCG}(\rs,P_f)$, and $\alpha^{-1}\beta\alpha^2\beta^{-1}\in\pi_1(\rs\setminus\Theta,z_0)$ are all identified, and
\begin{eqnarray} 
-\frac{41}{18} &=& B^{-1}A^2BA^{-1}.\frac{1}{0}\nonumber \\
&=& \frac{1}{0}.T_{\alpha}^{-1}\cdot T_{\beta}\cdot T_{\alpha}^2\cdot T_{\beta}^{-1} \nonumber \\
&=& \frac{1}{0}.\alpha^{-1}\beta\alpha^2\beta^{-1}\nonumber
\end{eqnarray}

We now describe a general procedure to draw curves in $\rs\setminus P_f$ corresponding to any $\frac{p}{q}\in\Qbar$.  It is easy to use the generators of the fundamental group to find curves corresponding to $\{\frac{0}{1}, \frac{1}{0}, -\frac{1}{1}\}$ as demonstrated in Figure \ref{fig:pointpushing}.  The algorithm above produces a word $w$ and an element $*\in\{\frac{0}{1}, \frac{1}{0}, -\frac{1}{1}\}$.  The element $w$ corresponds to an element of the fundamental group as was just described.  Record the effect on the curve corresponding to $*$ of pushing $z_0$ along this element of the fundamental group.  
\begin{figure}[h]
\centering
\includegraphics[width=120mm]{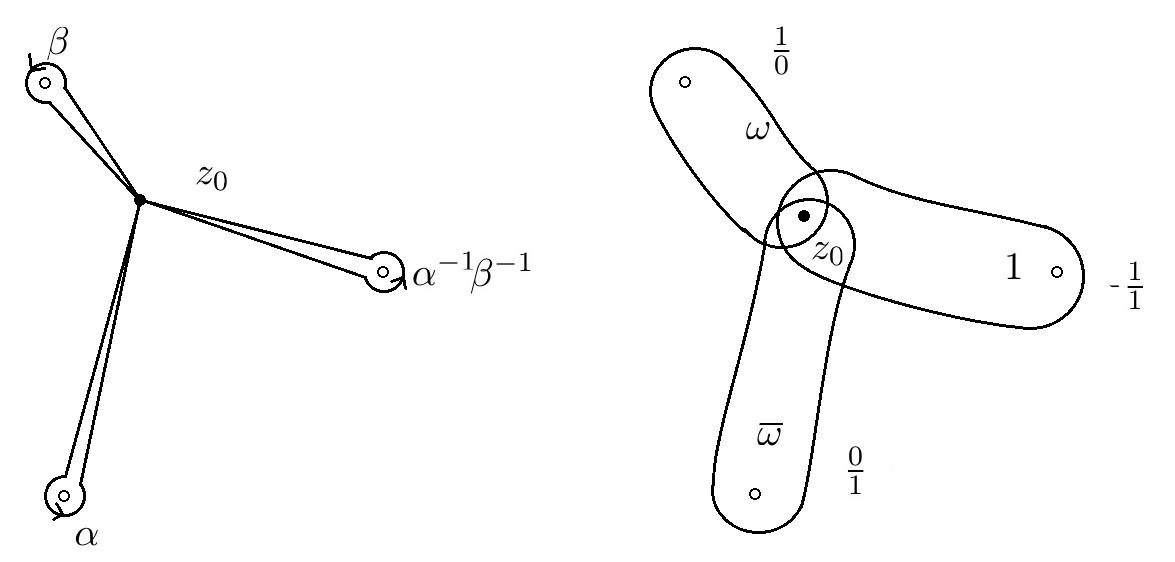}
\caption{Generators of fundamental group; three curves corresponding to points in orbit transversal}
\label{fig:pointpushing}
\end{figure} 
Figure \ref{fig:TwoOverOne} depicts the curve corresponding to $\frac{2}{1}$ which is found by pushing $z_0$ along $\beta$ and recording the effect on the curve corresponding to $\frac{0}{1}$.
\begin{figure}[h]
\centering
\includegraphics[width=55mm]{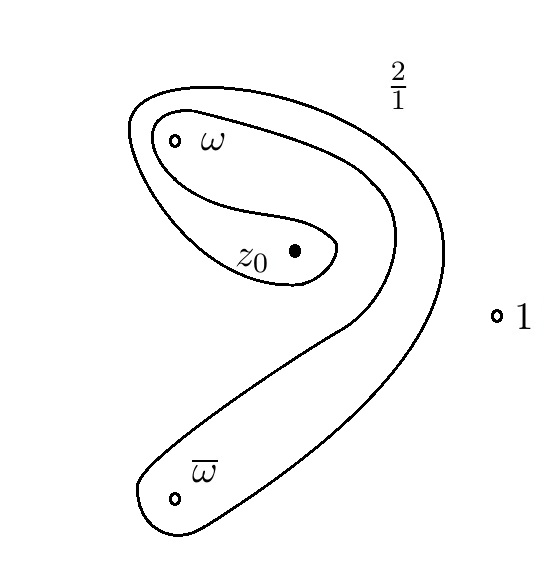}
\caption{The curve corresponding to $\frac{2}{1}$ found by pushing the curve corresponding to $\frac{0}{1}$ by $\beta$}
\label{fig:TwoOverOne}
\end{figure}

\end{comment}

\section{Slopes of Curves in $\rs\setminus P$ when $|P|=4$}
\label{sec:slope}

Section \ref{sec:FourPostcriticalPoints} presented a bijection between $\Qbar$ and the curves in $\rs\setminus P$ using the Weil-Petersson boundary of Teichm\"uller space.  A second natural way to assign an element of $\Qbar$ to such a curve is to compute its ``slope."  First normalize the four points by M\"obius conjugation to be $z_0,1,\omega,\bar{\omega}$.  We define the notion of slope for any $z_0$, and then give special consideration to the case $z_0=0$ where $P$ is just the postcritical set of the rational map $f$ considered later.  In this latter case, we relate the slope to the first assignment of extended rationals.

\subsection{Slopes in $\rs\setminus P$ when $|P|=4$}

To define slope on $\rs\setminus P$ where $P=\{z_0,1,\omega,\bar{\omega}\}$, we fix two minimally intersecting curves which are declared to have slope $\frac{0}{1}$ and $\frac{1}{0}$.  Connect $z_0$ and $\omega$ by a simple arc that avoids $P$ except at its endpoints.  Denote by $a$ the curve which is the boundary of a simply connected neighborhood of this arc.  Connecting $z_0$ and $\bar{\omega}$ by a simple arc that doesn't intersect $a$ or $P$, one uses a similar procedure to produce $b$.  The curves $a$ and $b$ in $\rs\setminus P$ are declared to have
slope $\frac{0}{1}$ and $\frac{1}{0}$ respectively.  One could chose an orientation for these curves by assuming for example that the point at infinity lies to the right; however, it will be evident in the following work that the definition of slope is independent of this choice of orientation.

Up to isomorphism, there exists a canonical double cover $(\mathbb{T},e_0)\longrightarrow(\widehat{\mathbb{C}},0)$ of $\rs$ branched over the four points in
$P$ where $e_0\in \mathbb{T}$ is the preimage of $0$.  This cover is unique up to isomorphism in the branched sense.  Denote by $\tilde{P}$ the set of preimages of $P=\{z_0,1,\omega, \bar{\omega}\}$
under the double cover where clearly $|\tilde{P}|=4$.  Also, denote by $(\mathbb{C},0)\longrightarrow(\mathbb{T},e_0)$ the universal cover, though an explicit choice of covering class representative will not be fixed until later.  The symbol $\simeq$ indicates that two curves in a surface are homotopic in that surface.

\begin{lemma}
\label{lem:torusLifts} The curves $a$ and $b$ in $\rs\setminus P$ each lift to pairs of
oriented simple closed curves $\tilde{a}_1, \tilde{a}_2$ and
$\tilde{b}_1, \tilde{b}_2$ in $\mathbb{T}\setminus \tilde{P}$
where $\tilde{a}_1 \simeq \tilde{a}_2$ and $\tilde{b}_1 \simeq
\tilde{b}_2$ in $\mathbb{T}$.
\end{lemma}

\begin{proof}
Let $g_1,g_2,g_3,g_4$ be small positively oriented loops about each
point in $P$, where the four homology classes $[g_i]$ evidently generate $H_1(\widehat{\mathbb{C}}\setminus{P},\mathbb{Z})$.  In other words, $H_1(\widehat{\mathbb{C}},\mathbb{Z})=\{\sum{c_i
[g_i]}|c_i\in\mathbb{Z}\}$.  A
simple closed curve in $\rs\setminus{P}$ lifts to
the double cover if it lies in the kernel of the homomorphism
$\rho:H_1(\widehat{\mathbb{C}},\mathbb{Z})\longrightarrow \mathbb{Z}/{2\mathbb{Z}}$ defined by
$$\rho(\sum{c_i [g_i]})=\sum{c_i} (\text{mod 2}).$$
Since the curves $a$ and $b$ are nonperipheral, they must bound a
disc containing two of the loops, so without loss of generality, one can say that $a$ is homologous to $\pm(g_1+g_2)$.  Since
$$\rho(a)=\rho(\pm(g_1+g_2)) =\pm(1+1) \equiv 0$$
it is clear that $a$ lifts to $\mathbb{T}$.  A similar argument applies to $b$.  Since the cover is two-to-one, $a$ lifts to two disjoint curves $\tilde{a}_1,\tilde{a}_2$ in $\mathbb{T}$, and $b$ lifts to two disjoint curves $\tilde{b}_1,\tilde{b}_2$.  Note that $\tilde{a}_1$ and $\tilde{b}_1$ intersect minimally in $\rs\setminus P$ in the sense that
\[\min_{a'\simeq \tilde{a}_1,b'\simeq \tilde{b}_1}|a'\cap b'|=1\] 
because
\[\min_{a'\simeq a,b'\simeq b}|a'\cap b'|=2.\]
So in fact these the curves $\tilde{a}_i$ are not homotopic to the curves $\tilde{b}_i$ in $\mathbb{T}$.

We prove next that $\tilde{a}_1 \simeq \tilde{a}_2$ and $\tilde{b}_1 \simeq
\tilde{b}_2$ in $\mathbb{T}$.  The
nontrivial deck transformation $h:\mathbb{T}
\rightarrow \mathbb{T}$ interchanges $\tilde{a}_1$ with
$\tilde{a}_2$ and it interchanges $\tilde{b}_1$ with $\tilde{b}_2$.  On the level of
the universal cover $(\mathbb{C},0)\rightarrow (\mathbb{T},e_0)$, choose the lift of $h$ that fixes the origin.  The only involutive holomorphic deck transformation of the torus cover that fixes the origin is the map $z\mapsto-z$.  Thus, the induced map on homology is $h_*:H_1(\mathbb{T},\mathbb{Z})\rightarrow
H_1(\mathbb{T},\mathbb{Z})$ given by $(x,y)\mapsto(-x,-y)$.  But two curves in the torus whose homology classes agree up to sign are freely homotopic in the torus. 
\qed
\end{proof}

We now define precisely the notion of slope in $\mathbb{T}$.  Since $\tilde{a}_1$ and $\tilde{b}_1$ form an ordered basis of $H_1(\mathbb{T},\mathbb{Z})$, the assignment
\[\tilde{a}_1\mapsto (1,0)\]
\[\tilde{b}_1\mapsto (0,1),\]  
gives a natural identification $H_1(\mathbb{T},\mathbb{Z})\cong\mathbb{Z}\oplus\mathbb{Z}$.  Slopes are assigned to homology classes by the map \[(q,p)\mapsto \frac{p}{q}\] where $p$ and $q$ are relatively prime.  Define a \textit{curve of slope $\frac{p}{q}$ in $\mathbb{T}$} to be a simple closed curve lying in the homology class corresponding to $(q,p)$.  Suppose that the preimage of an essential curve in $\rs\setminus P_f$ is two essential homotopic curves in $\mathbb{T}$ with homology classes $(q,p)$ and $(-q,-p)$. These curves are both assigned a slope of $\frac{p}{q}$ in the torus.  A curve in $\rs\setminus P$ is said to be a \textit{curve of slope $\frac{p}{q}$ in $\rs\setminus P$} if it lifts to a curve of slope $\frac{p}{q}$ in the torus.

It is a standard fact that the cup product yields a signed intersection form for essential curves in $\mathbb{T}$, and this can be used to give another interpretation of slope in $\mathbb{T}$. Let $a$ and $b$ be two oriented curves in an oriented surface $X$. Apply a homotopy so the curves intersect minimally at $N$ distinct points $x_1,...,x_N$.  Then the \textit{signed intersection number} of $a$ and $b$ is defined by $\iota(a,b)=\sum_{j=1}^{N} sgn(x_j)$ where the sign function $sgn(x_j)$ is computed using the conventions below.
\begin{figure}[h]
\centerline{\includegraphics[width=55mm]{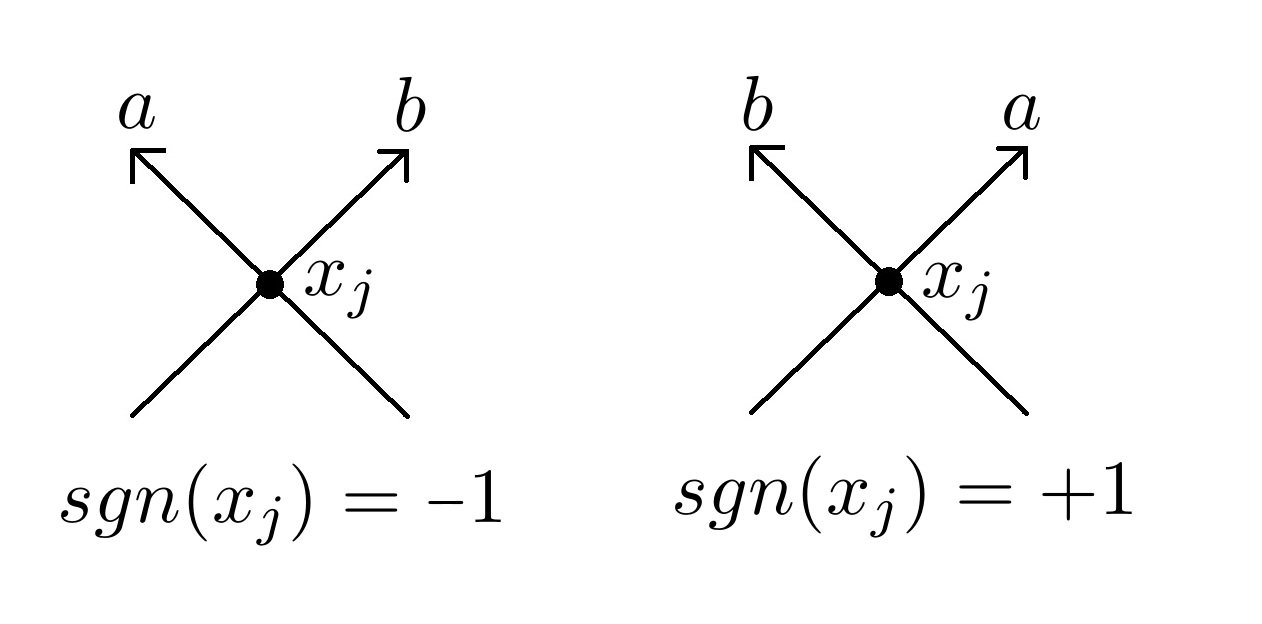}}
\label{fig:sigma}
\end{figure}

A straightforward cup product computation shows that the signed intersection number for curves of slope $\frac{p}{q}$ and $\frac{r}{s}$ in $\mathbb{T}$ is the determinant of $\left[\begin{smallmatrix} p&r\\ q&s \end{smallmatrix}\right]$.  One can then see that for relatively prime $p$ and $q$, a curve of slope $\frac{p}{q}$ will have $p$ signed intersections with $\tilde{a}_1$ and $q$ signed intersections with $\tilde{b}_1$.

%One can conclude that $\iota(\tilde{a}_1,\tilde{b}_1)=\iota(\tilde{a}_2,\tilde{b}_2)$ because of %the previous discussion of the nontrivial deck element $h$.  

Next we define slope in $\mathbb{C}$ with respect to the lattice $\Lambda$.  Note that $H_1(\mathbb{T},\mathbb{Z})\cong \pi_1(\mathbb{T},e_0)$, and so the basis formed by $\tilde{a_1}$ and $\tilde{b_1}$ act on the universal cover $(\mathbb{C},0)\rightarrow(\mathbb{T},e_0)$ by translation; explicitly choose the universal cover so that the translation corresponding to $a$ is $z\mapsto z+1$ and the one corresponding to $b$ is $z\mapsto z+ \tau$ for some nonreal complex number $\tau$. The torus $\mathbb{T}$ can then be thought of as the quotient $\mathbb{C}/\Lambda$ where $\Lambda=\langle 1,\tau \rangle$. %Thus, the lifts of $\tilde{a}_1$ and $\tilde{b}_1$ to the universal cover can be homotoped rel $\Lambda$ to straight line segments,
%where the homotoped lift of the $0/1$ curve has slope $0$ in the usual sense, and the lift of the $1/0$ has slope $\text{Im}\tau/\text{Re}\tau$. 
The following defines \textit{the line of slope $\frac{p}{q}$ in $(\mathbb{C},\Lambda)$} where $t\in\mathbb{R}$ and $c_0\in\mathbb{C}$ is chosen so that the line avoids the lattice $\Lambda$:
$$t\mapsto t(p\cdot\tau+q)+c_0.$$
Pushing this curve down to
$\widetilde{\mathbb{C}}\setminus{P}$ and forgetting the
orientation on the curve yields a curve of slope $\frac{p}{q}$ in
$\widetilde{\mathbb{C}}\setminus{P}$ as defined before.

\subsection{Slopes in $\rs\setminus P$ when $P=\{0,1,\omega,\bar{\omega}\}$}

Now we apply the methods described above to the highly symmetric case when $P=P_f$.  The choice of $\frac{0}{1}$ and $\frac{1}{0}$ curves in $\rs\setminus P$ are exhibited in Figure \ref{fig:lattice}, and following the steps above, the map $\pi:\mathbb{C}\rightarrow\widehat{\mathbb{C}}$ is produced.  Because of the symmetry of $P$, a convenient triangulated model of $\pi$ exists, and furthermore, one can explicitly find a formula for $\pi$ in terms of the Weierstrass function.

The combinatorial model of $\pi$ is produced by mapping a certain triangle in $\mathbb{C}$ with label $3$ by isometry to the top face of the tetrahedron from Section \ref{sec:dynamicalPlaneFacts}.  Namely:
$$0 \mapsto B,\hspace{.5 cm} \frac{1}{2}  \mapsto C, \hspace{.5cm} \frac{1}{4}+\frac{\sqrt{3}}{4}\cdot i \mapsto D.$$
Extend the map over $\mathbb{C}$ using reflection.  Having produced the combinatorial model, postcompose this map to the tetrahedron by the radial projection and then the stereographic projection with north pole $N$ as in Section \ref{sec:dynamicalPlaneFacts} to obtain $\pi$.  This $\pi$ is evidently a
meromorphic function with double zeros at each point in the lattice
$\Lambda = \langle 1,\tau:=e^{2\pi\cdot i/3} \rangle$.  Then on $\mathbb{C}$, the function $\frac{1}{\pi}$ is doubly periodic with respect to
$\Lambda$ and has double poles at each point in $\Lambda$.  Denote by $\wp_{\Lambda}$ the Weierstrass function with lattice
$\Lambda$, i.e.
$$\wp_{\Lambda}(z)=\frac{1}{z^2}+\sum_{\lambda\in\Lambda}\frac{1}{(z-\lambda)^2}-\frac{1}{\lambda}.$$   
Liouville's theorem can be used to prove that there exists a number $c\in\mathbb{C}$ such that
$1/{\pi(z)}=c\wp_\Lambda(z)$.  To find the explicit value of $c$, note that $\pi(1/2)=1$ from an examination of Figure \ref{fig:lattice}, and so $c=\frac{1}{\wp_{\Lambda}(1/2)}$.  Thus:
$$\pi(z)=\frac{\wp_{\Lambda}(1/2)}{\wp_{\Lambda}(z)}.$$

We describe a simple way to produce a line of slope $\frac{p}{q}$ in $\widehat{\mathbb{C}}\setminus
P_f$.  
\begin{figure}[h!]
\centerline{\includegraphics[width=130mm]{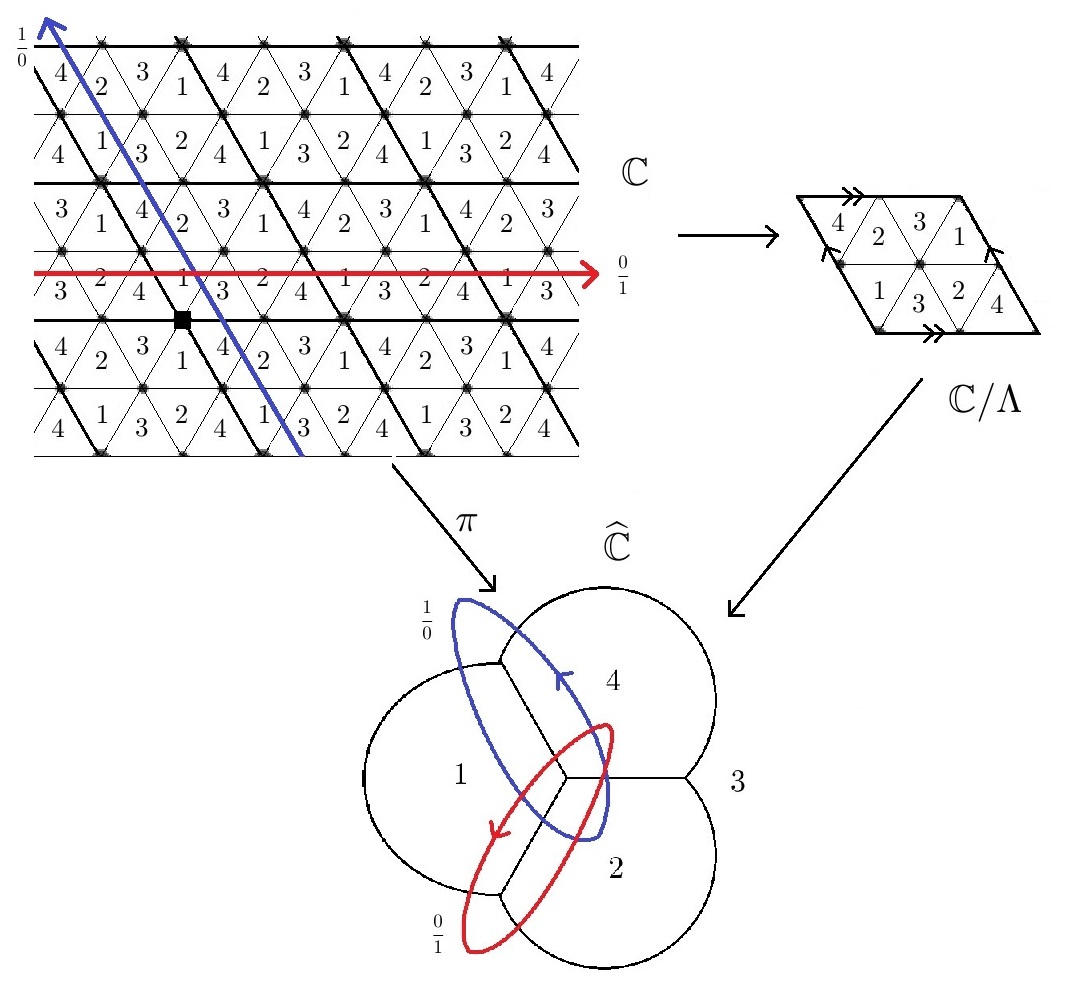}}
\caption{The map $\pi:(\mathbb{C},\Lambda)\rightarrow(\rs,0)$ used to compute slopes in $(\rs,P_f)$.  Larger dots indicate points in $\Lambda$, with the exception of the origin which is marked by a filled box.  The vertices of the small triangles are the half-lattice points of $\Lambda$.}
\label{fig:lattice}
\end{figure}
The line of slope $\frac{p}{q}$ in $(\mathbb{C},\Lambda)$ is
$$t\mapsto c_0+t(p\cdot e^{2\pi\cdot i/3}+q\cdot t), t\in\mathbb{R}$$
for a generic $c_0\in\mathbb{C}$.  Projecting this line to $\widehat{\mathbb{C}}\setminus P_f$ gives a
curve of slope $\frac{p}{q}$ in $\widehat{\mathbb{C}}\setminus P_f$.  Curves of slope $\frac{1}{0}$ and $\frac{0}{1}$ in both $(\mathbb{C},\Lambda)$ and $\rs\setminus P_f$ are depicted in Figure \ref{fig:lattice}.  It can be observed that a curve of slope $\frac{p}{q}$ in $\rs\setminus P_f$ lies in the same homotopy class as the curve corresponding to the point $-\frac{p}{q}$ in the Weil-Petersson boundary (this latter assignment is discussed in Section \ref{sec:FourPostcriticalPoints} and made explicit in Section \ref{fProperties}).

\begin{comment}
and connect it with the ideas of slope with respect to a lattice in $\mathbb{C}$ and slope of curves in $\rs\setminus P$.   First we fix a choice of what is meant by a curve of slope $\frac{0}{1}$ and $\frac{1}{0}$ in the setting of the torus, consistent with the choices of such a curve already made in $\rs\setminus P$. An intersection number is then defined that allows one to define the curve of slope $\frac{p}{q}$ in $\mathbb{T}$.  Next, an explicit procedure is described to fix the covering map $(\mathbb{C},0)\rightarrow(\mathbb{T},e_0)$ in the spirit of Figure \ref{fig:lattice}, and then the definition of the curve of slope $\frac{p}{q}$ in $\mathbb{C}$ with respect to $\Lambda$ is made.  Finally, these curves can be pushed down to $\rs\setminus P$ to define slope in that setting as well.
\end{comment}

\end{comment}
\section{Analysis of a specific example: $f(z)=\frac{3z^2}{2z^3+1}$}
\label{fProperties}

 Some basic properties of $f$ are presented, and then a study is made of the moduli space map.  A wreath recursion is produced for the map on moduli space as well as for $f$.
\subsection{The Dynamical Plane of $f$}\label{sec:dynamicalPlaneFacts}
The postcritical set of $f$ is $P_f=\{0,1,\omega,\overline{\omega}\}$, where each point is also critical.  The critical portrait of $f$ is:
\begin{align*}
\centerline{\xymatrix{\omega\ar@/^/[d]^{\times 2} &0\ar@(dl,dr)_{\times 2} &1\ar@(dl,dr)_{\times 2}  \\ \overline{\omega}\ar@/^/[u]^{\times 2}}}
\end{align*}
so $f$ is a nearly Euclidean Thurston map \cite{CFPP2}.  Let $\Theta=\{1,\omega,\overline{\omega}\}$ as before.

First a finite subdivision rule for $f$ is presented.  Consider the two big equilateral triangles in Figure \ref{fig:fsubdivision} as subsets of $\mathbb{R}^2$ where both triangles have vertices at $(0,0),(2,0),$ and $(1,\sqrt{3})$.
\begin{figure}[h]
\centerline{\includegraphics[width=130mm]{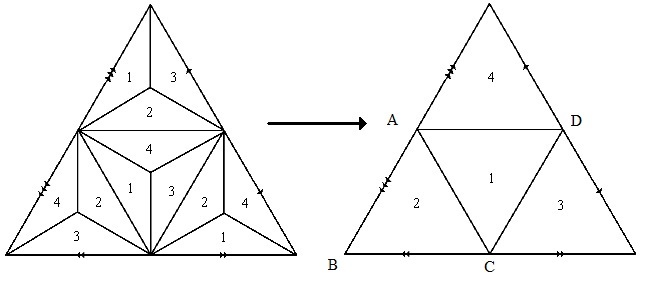}}
\caption{Affine model of $f$}
\label{fig:fsubdivision}
\end{figure}
 Identify the marked edges to form two tetrahedra.  Define the affine model by first mapping the triangle from the domain with vertices $(1,0)$ and $(2,0)$ and label $1$ to the triangle in the range labeled $1$ by the unique affine map that respects the labels of adjacent triangles.  Extend this map by reflection over the whole tetrahedron to produce a three to one map from the tetrahedron to itself with four simple critical points.  To prove that this is actually a model of $f$, explicitly embed this tetrahedron into $\mathbb{R}^3$ isometrically by mapping:
\[A\mapsto (0,0,-\frac{\sqrt{18}}{3}) \hspace{2cm} B\mapsto (-\frac{1}{2},\frac{\sqrt{3}}{2},0)\]
\[C\mapsto (1,0,0) \hspace{2 cm} D\mapsto (-\frac{1}{2},-\frac{\sqrt{3}}{2},0).\]
Circumscribe this regular tetrahedron by a sphere and project radially.  The stereographic projection from this sphere with north pole $N$ at $(0,0,\frac{1}{\sqrt{2}})$ onto the plane $\{(x,y,z):z=0\}$ is conformal.  Identify this plane with the complex plane by the mapping $(x,y,z)\mapsto x+iy$.  The composition of these maps yields a conformal map from the tetrahedron onto $\rs$.  Conjugating the affine model for $f$ by this conformal map yields (up to homotopy) a rational function that has local degree 2 at each point in $P_f$.  This must be $f$ by uniqueness.

The function $f$ has appeared in the literature before: it is shown in \cite{BEKP} that the Thurston pullback map $\sigma_f$ is surjective and has a fixed point of local degree 2.  Up to some non-dynamical equivalence, this is the only known example where the pullback map is surjective.  An image generated by Xavier Buff using the method described in Section \ref{sec:IMG} seems to indicate that under iteration of the extended Thurston pullback map, each boundary point maps to one of three points.  Thus, using ideas from Selinger's work, there is reason to believe that under iterated preimage of $f$, all curves eventually land in the homotopy class of one of three essential curves.

Another feature of $f$ is that it arises as a mating \cite{Mi00} of the cubic polynomial $P$ with two finite fixed critical points and $Q$ which interchanges its two finite critical points.  This follows from a later section where it will be seen that the curve corresponding to $-\frac{1}{1}$ pulls back to itself by degree three in an orientation preserving way, and since $f$ is hyperbolic, it is a mating with equator given by this invariant curve \cite{Me11}. Note that the polynomials $P$ and $Q$ are unique up to affine conjugacy.  This implies that $P$ commutes with the map interchanging its two critical points, and the same holds for $Q$.  It follows that though there are two different ways of identifying the circles at infinity for $P$ and $Q$, the result is $f$ either way.

\subsection{The Correspondence on Moduli Space}

Recall the correspondence on moduli space mentioned in Section \ref{sec:notationanddefns}.  Note that $Y$ is the unique degree 4 rational function that fixes each of $1,\omega,$ and $\overline{\omega}$ with local degree 3, because for another such $W$, deg$(Y-W)\leq 8$, but $Y-W$ has 9 zeros counted with multiplicity.  Though the dynamical properties of $Y$ aren't important here, the map $Y$ was used by McMullen \cite{McM} to give a generally convergent iterative algorithm to solve cubic equations.

We now produce the affine model for $Y$ depicted in Figure \ref{fig:Y}.  The domain of the model is the equilateral triangle with vertices $-2,1\pm\sqrt{3}i$ doubled over its boundary.  The range is the equilateral triangle with vertices $1,\omega,\overline{\omega}$ doubled over its boundary.  Thus the domain is the union of eight equilateral triangles with disjoint interior, where the three outer triangles on the front face are shaded as well as the central triangle on the back face.   The range is the union of two equilateral triangles with disjoint interior, where the back face is shaded. The affine model for $Y$ which we denote by $Y^{\triangle}$ is constructed by mapping the unshaded triangle on the front face of the domain to the front face of the range by the identity.  Extend this map by reflection to the whole domain.  This produces a degree 4 map that is conformal except at the vertices $1,\omega,\overline{\omega}$ which all map by degree 3.  

\begin{figure}[h]
\centerline{\includegraphics[width=110mm]{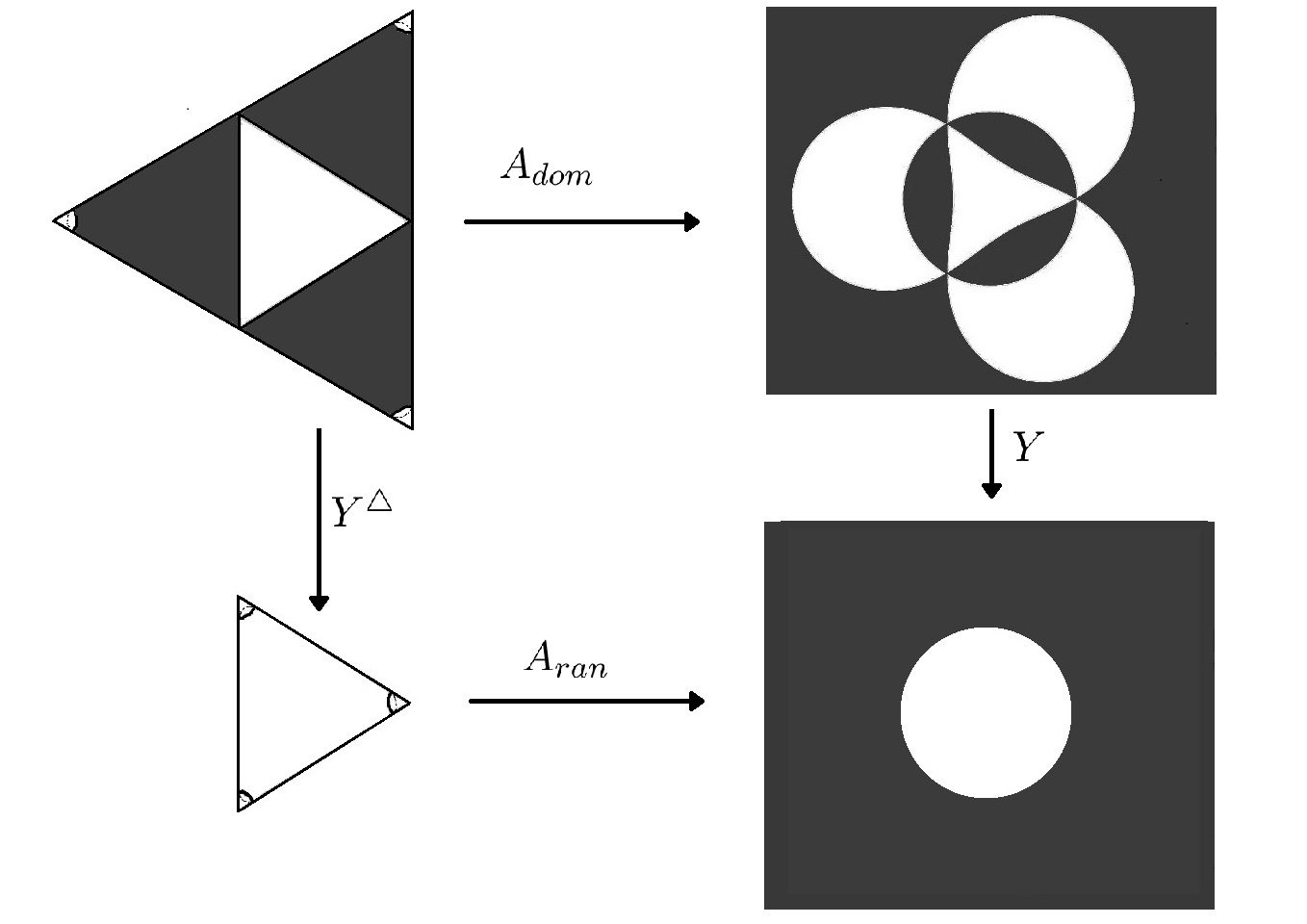}}
\caption{Isomorphisms from domain and range of affine model to $\rs$}
\label{fig:Y}
\end{figure}

Let $\triangle$ be the Euclidean triangle in $\mathbb{C}$ with vertices $1,\omega,$ and $\overline{\omega}$. There is a unique Riemann map $(\triangle,\{1,\omega,\overline{\omega}\})\rightarrow(\mathbb{D},\{1,\omega,\overline{\omega}\})$.  The map $A_{ran}$ in Figure \ref{fig:Y} is the unique isomorphism to $\rs$ defined by applying Schwarz reflection to this Riemann map.  $A_{ran}$ evidently fixes the points $1,\omega,\overline{\omega}$.  There is also a unique Riemann map from the front face of the domain of $Y^{\triangle}$ to the unit disk in $\rs$ that fixes the points $1,\omega,\overline{\omega}$.  The map $A_{dom}$ is the unique isomorphism to $\rs$ defined by applying Schwarz reflection to this Riemann map.  It is clear that $A_{ran}^{-1}\circ Y^{\triangle}\circ A_{dom}^{-1}$ is a holomorphic map on $\rs$ and it must be identically $Y$ by uniqueness.

The discussion of Section \ref{sec:FourPostcriticalPoints} is now specialized to the case where $P_f = \{0,1,\omega,\overline{\omega}\}$.  Fix $\pi$ to be the universal cover $(\mathbb{H},\frac{1+\sqrt{3}i}{2})\rightarrow(\rs\setminus\Theta,0)$ depicted in Figure \ref{fig:cover}.  The deck group is again $\text{P}\Gamma(2)=\langle A,B \rangle$ where $A = \left[
\begin{smallmatrix} 1&0\\ -2&1 \end{smallmatrix}\right]$, $B =\left[
\begin{smallmatrix} 1&2\\ 0&1 \end{smallmatrix}\right]$, and these generators are identified with the generators $\alpha$ and $\beta$ respectively of the fundamental group $\pi_1(\rs\setminus\Theta,0)$ which are also depicted in Figure \ref{fig:cover}.  Pushing the point $0$ in the positive direction along $\alpha$ and $\beta$ yield generators of $\text{PMCG}(\rs,P_f)$ which are denoted $T_\alpha$ and $T_\beta$ respectively.

\subsection{The Virtual Endomorphism and Wreath Recursion on Moduli Space}

The Schreier graph and Reidemeister-Schreier algorithm from Section \ref{subsec:ReidemeisterSchreier} are used to compute the virtual endomorphism on moduli space.  Figure \ref{fig:generatorlifts} exhibits the preimages of the two generators $\alpha$ and $\beta$ under $Y$.  The basepoint is chosen to be the origin in both domain and range.
\begin{figure}[h]
\includegraphics[width=120mm]{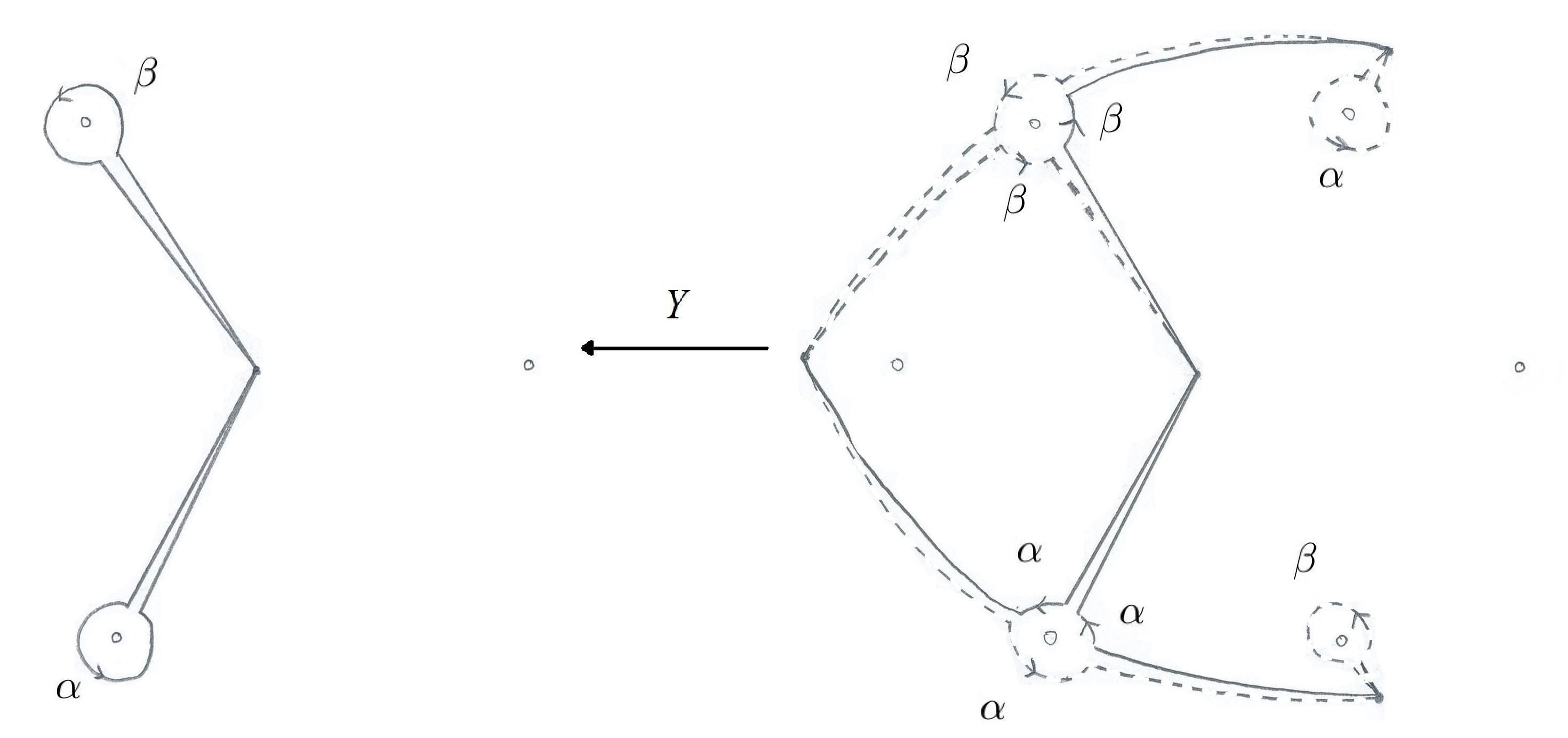}
\caption{ Lift of generators under $Y$; the Schreier graph}
\label{fig:generatorlifts}
\end{figure}
The choice of maximal subtree is indicated by solid lines, and the corresponding Schreier transversal is $T=\{1,\alpha,\alpha^{-1},\beta^{-1}\}$.  From Figure \ref{fig:generatorlifts} and the discussion in \ref{subsec:ReidemeisterSchreier},
\[H=\langle\beta\alpha\beta^{-1},\beta^2\alpha^{-1},\beta^{-1}\alpha^{-1},\alpha^3,\alpha^{-1}\beta\alpha\rangle.\]
To define the virtual endomorphism on moduli space, it is necessary to write an arbitrary $w\in H$ as a word in the five generators of $H$ and their inverses; the Reidemeister-Schreier rewriting process will be used, where $S=\{\alpha,\beta\}$.  Table \ref{Reide} exhibits all necessary values of $\gamma$, where the left column corresponds to the first argument of $\gamma$, and the top row corresponds to the second argument.
\begin{table}
\renewcommand{\arraystretch}{1.25}
\centerline{\begin{tabular}{ | c || c | c | c | c | }
  \hline
  $\gamma(column,row)$&$\alpha$ & $\alpha^{-1}$& $\beta$& $\beta^{-1}$\\ \hline\hline
  $1$&$1$ & $1$& $1$& $\beta^{-1}\alpha^{-1}$\\ \hline
  $\beta$&$\beta\alpha\beta^{-1}$&$\beta\alpha^{-1}\beta^{-1}$ & $\beta^2\alpha^{-1}$& $1$\\ \hline
  $\alpha$&$\alpha^3$&$1$&$\alpha\beta$ & $\alpha\beta^{-2}$\\ \hline
  $\alpha^{-1}$&$1$&$\alpha^{-3}$&$\alpha^{-1}\beta\alpha$ & $\alpha^{-1}\beta^{-1}\alpha$\\ \hline
\end{tabular}}
\caption{All values of $\gamma$ in the Reidemeister-Schreier algorithm}
\label{Reide}
\end{table}
Recall that path multiplication was earlier defined so that for example $\alpha\beta$ is the path obtained by traversing $\alpha$ in the positive direction followed by $\beta$.  The following is a sample computation to show how one would write the word $\alpha\beta^2\alpha^{-1}\beta^{-1}\in H$ in terms of the generators of $H$:
\begin{eqnarray}
\alpha\beta^2\alpha^{-1}\beta^{-1}&=& \gamma(1,\alpha)\cdot\gamma(\overline{\alpha},\beta)\cdot\gamma(\overline{\alpha\beta},\beta)\cdot\gamma(\overline{\alpha\beta^2},\alpha^{-1})\cdot\gamma(\overline{\alpha\beta^2\alpha^{-1}},\beta^{-1})\nonumber\\
&=&\gamma(1,\alpha)\cdot\gamma(\alpha,\beta)\cdot\gamma(1,\beta)\cdot\gamma(\beta,\alpha^{-1})\cdot\gamma(\beta,\beta^{-1})\nonumber\\
&=&1\cdot\alpha\beta\cdot 1\cdot\beta\alpha^{-1}\beta^{-1}\cdot 1\nonumber
\end{eqnarray}

This algorithm describes how to lift elements of the fundamental group under $Y$ based at $0$ but in order to compute $\phi_f:H<\pi_1(\rs\setminus \Theta,0)\rightarrow\pi_1(\rs\setminus \Theta,0)$ as in Section \ref{sec:npoints}, one needs to compute the image of these lifts under $X(z)=z^2$.  It can be shown that $\phi_f$ behaves on generators as follows:
\[\beta\alpha\beta^{-1}\longmapsto\beta\]
\[\beta^2\alpha^{-1}\longmapsto\beta^{-1}\]
\[\beta^{-1}\alpha^{-1}\longmapsto\alpha^{-1}\beta^{-1}\]
\[\alpha^3\longmapsto\beta\]
\[\alpha^{-1}\beta\alpha\longmapsto\alpha.\]
The whole process of computing $\phi_f(w)$ for some word $w=s_1s_2...s_k\in H$ where $s_i\in S\cup S^{-1}$ is streamlined in Figure \ref{fig:VEmachine}.  The starting state is the vertex labeled 1.  Let $j=1$, and the output string is initially empty.

If $s_j\in S$, the new state is determined by following the arrow with the group element $s_j$ in the first coordinate of the label; append the contents of the second coordinate to the right of the output string.  If $s_j\in S^{-1}$, the new state is determined by following the arrow with $s_j^{-1}$ in the first coordinate; append the inverse of the second coordinate to the right of the output string.  Increment $j$ and repeat the process described in this paragraph until the whole input string $w$ is consumed.  Upon completion, the output string is precisely $\phi_f(w)$.  For example, the diagram gives that 
\begin{eqnarray}
\phi_f(\alpha^2\beta^{-1}\alpha^{-1}\beta^{-1}\alpha^2\beta^{-1})&=& 1 \cdot\beta\cdot\alpha^{-1}\cdot\beta^{-1}\cdot\beta\cdot\beta\cdot\beta\cdot 1\nonumber\\
&=& \beta\alpha^{-1}\beta^2.\nonumber
\end{eqnarray}
\begin{figure}[h]
\centerline{\includegraphics[width=70mm]{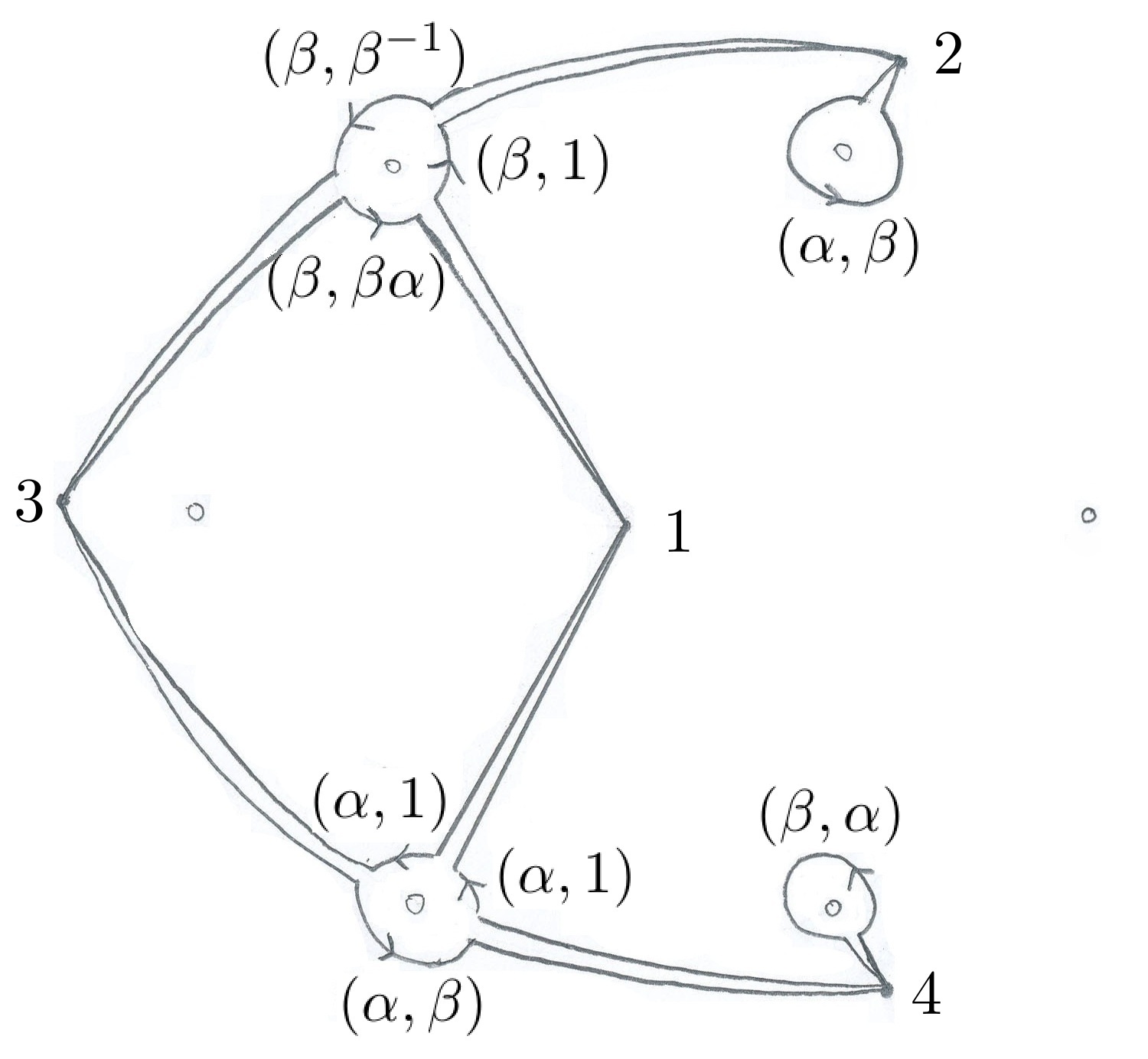}}
\caption{Machine to compute the virtual endomorphism $\phi_f$ drawn in $\mathcal{W}_f=\rs\setminus\Theta'$}
\label{fig:VEmachine}
\end{figure}

The contraction of wreath recursions have been used in the past to better understand the virtual endomorphism on moduli space \cite[Thm 1.4]{Pil10}.  We define a wreath recursion on $\pi_1(\rs\setminus\Theta)$ but show that it is not contracting.
\begin{figure}[h]
\centerline{\includegraphics[width=60mm]{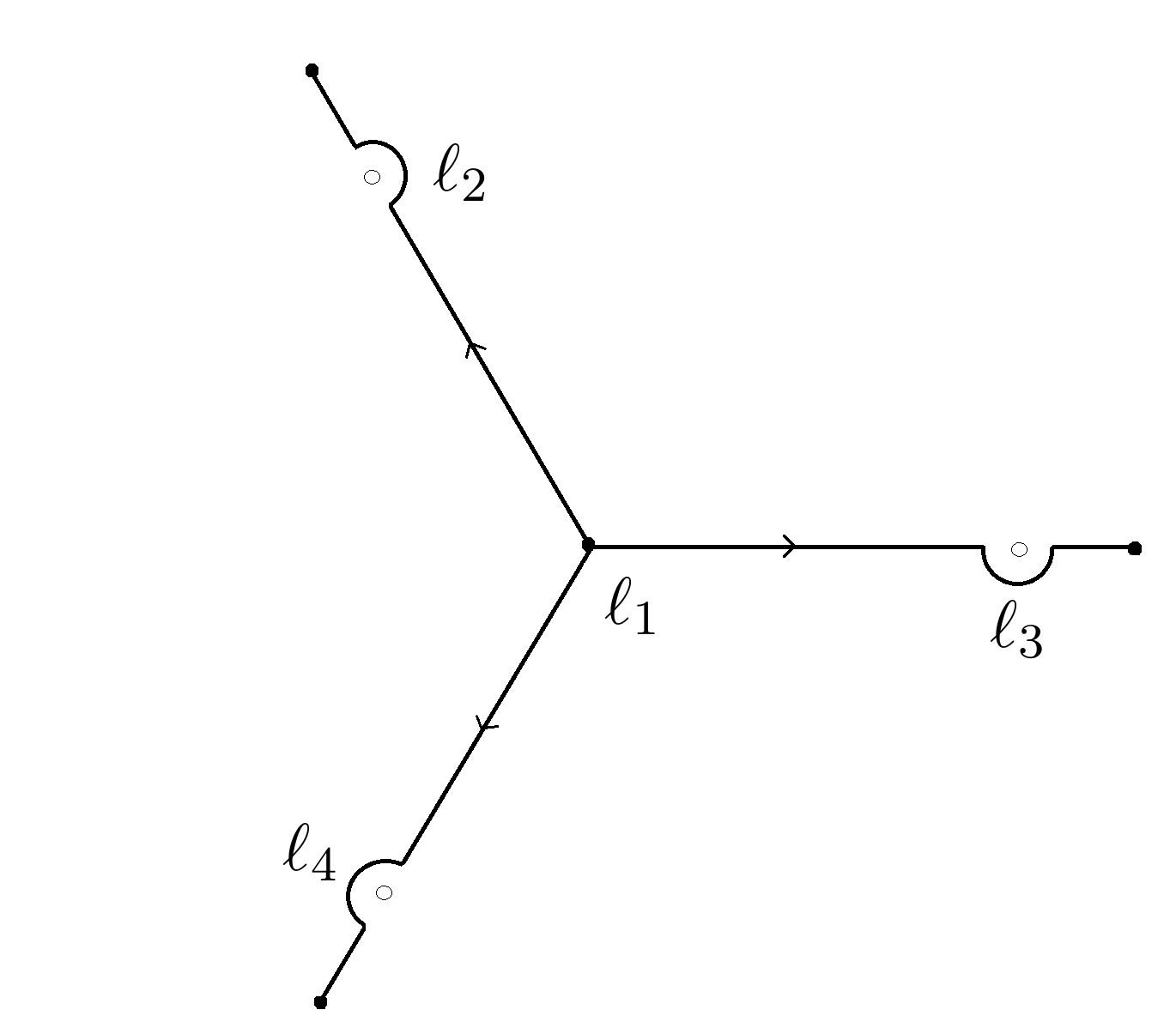}}
\caption{Connecting paths for the wreath recursion $\Phi$ drawn in $\rs\setminus\Theta$}
\label{fig:connectingpaths}
\end{figure}
The three small circles in Figure \ref{fig:connectingpaths} denote the points in $\Theta$, and the four dots correspond to the points in $X\circ Y^{-1}(0)=\{0,\sqrt[3]{4},\sqrt[3]{4}\omega,\sqrt[3]{4}\overline{\omega}\}$ which are labeled 1,3,2, and 4 respectively.  The connecting path $\ell_1$ is the constant path at $0$.  The path $\ell_3$ runs from $0$ to $\sqrt[3]{4}$ along the positive real axis, except in a small neighborhood of $1+0i$ where the path goes into the lower half-plane along a small semi-circular arc.  The other paths are $\ell_2=\omega\cdot\ell_3$ and $\ell_4=\overline{\omega}\cdot\ell_3$.  Given this choice of connecting paths,
\[\Phi(\beta)=\langle\langle \beta^{-1},\beta\alpha,1,\alpha \rangle\rangle(\,1\quad 2\quad 3\,)\]
\[\Phi(\alpha)=\langle\langle \beta\alpha, \beta, \alpha^{-1}, 1 \rangle\rangle(\,1\quad 3\quad 4\,).\]
Direct computation shows that
\[\Phi((\beta\alpha)^{3n})=\langle\langle (\beta\alpha)^n,(\beta\alpha)^n,(\beta\alpha)^{3n},(\alpha\beta)^n \rangle\rangle id.\]
Lemma 2.11.2 in \cite{Nek} implies that a candidate nucleus $\mathcal{N}$ for this wreath recursion must contain $(\beta\alpha)^{3n}$ for all $n$ since $((S\cup\mathcal{N})^2)|_{X^k}\subset\mathcal{N}$ for all $k\in \mathbb{N}$.  But since $\mathcal{N}$ would be infinite, the wreath recursion $\Phi$ is not contracting.

\subsection{The Wreath Recursion on the Dynamical Plane}

To compute the wreath recursion on the dynamical plane, let the basepoint be the unique real fixed point of $f$ that lies between 0 and 1.  The choice of four generators $\alpha,\beta,\gamma,$ and $\delta$ of $\pi_1(\rs\setminus P_f)$ is presented in Figure \ref{fig:fgeneratorlifts}, where $\alpha\gamma\beta\delta$ is trivial.   A choice of connecting paths $\ell_1,\ell_2,\ell_3$ is also made, where $\ell_2$ is the constant path at the basepoint.  The endpoint of each path $\ell_1,\ell_2,\ell_3$ is labeled $1,2,3$ respectively.  With these choices, the wreath recursion on the dynamical plane is
\begin{align*}
&\Phi_f(\alpha)=\langle\langle e,e,\beta \rangle\rangle (\,1\quad3\,)\\
&\Phi_f(\beta)=\langle\langle \beta^{-1},e,\gamma^{-1}\delta^{-1} \rangle\rangle (\,1\quad3\,)\\
&\Phi_f(\gamma)=\langle\langle e,\gamma,e \rangle\rangle (\,2\quad3\,)\\
&\Phi_f(\delta)=\langle\langle \delta,e,e \rangle\rangle (\,1\quad2\,).
\end{align*}

\begin{figure}[p]
\includegraphics[width=180mm]{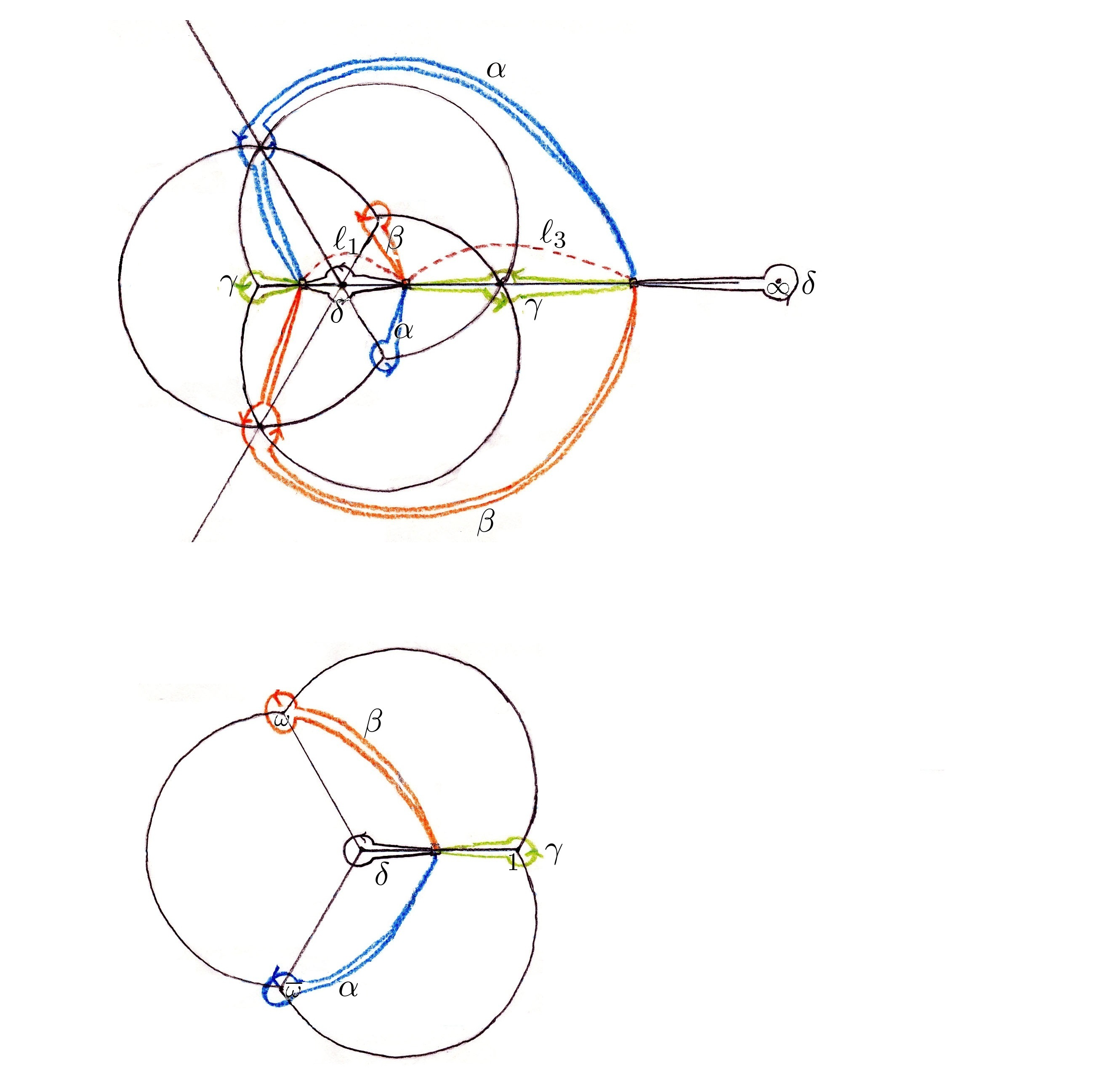}
\caption{Choice of connecting paths and lift of generators under $f$}
\label{fig:fgeneratorlifts}
\end{figure}

\end{comment}

\section{Boundary Values of $\sigma_f$}
\label{SigmaFGA}

In this section, we analyze the behavior of the extended Thurston pullback map $\sigma_f$ on the Weil-Petersson boundary of Teichm\"uller space for $f$.

\subsection{The Boundary Maps to the Boundary}

Denote the Weil-Petersson completion of Teichm\"uller space by $\overline{\mathscr{T}}_f$, and let $\partial\mathscr{T}_f$ denote the Weil-Petersson boundary.  We first show that the extended Thurston pullback map $\overline{\sigma}_f:\overline{\mathscr{T}}_f\rightarrow\overline{\mathscr{T}}_f$ has the property that $\overline{\sigma}_f(\partial\mathscr{T}_f)\subset\partial\mathscr{T}_f$.  This is accomplished by showing that the preimage under $f$ of an essential curve in $\rs\setminus P_f$ has an essential component.

To compute the preimages of essential curves in $\rs\setminus P_f$ under $f$, first identify such a curve with a parabolic element of $\pi_1(\rs\setminus\Theta,0)$ as discussed in Section \ref{sec:FourPostcriticalPoints}.  This element may not lie in $H$, so an appropriate power must be taken to ensure this.  Then the discussion in Section \ref{subsec:DefinitionsEndomorphisms} indicates that one should lift this element by $Y$ based at 0, and push it down by $X$.  This section is devoted to understanding this process.  For convenience the notation $\gamma=\alpha^{-1}\beta^{-1}$ is used.

\begin{lemma}\label{lem:nonperipheral}
Let $g\in\{w\alpha^n w^{-1}, w\beta^n w^{-1}, w\gamma^n w^{-1}\}\cap H$ where $n\in\mathbb{Z},w\in\pi_1(\rs\setminus\Theta,0).$  Then $X_*(Y^{-1}(g)[0])\neq 1$.
\end{lemma}
\begin{proof}
We sketch a topological proof of this fact.  Note that $g$ is freely homotopic to a peripheral curve about some point $p\in\Theta$.  Thus $Y^{-1}(g)[0]$ is peripheral in $\rs\setminus\Theta'$.  The function $X(z)=z^2$ maps peripheral curves in $\rs\setminus\Theta'$ to peripheral curves in $\rs\setminus\Theta$.\qed
\end{proof}

The fact that $\overline{\sigma}_f(\partial\mathscr{T}_f)\subset\partial\mathscr{T}_f$ is a consequence of this lemma and the following argument.  Let $\Gamma$ be an essential curve in $\rs\setminus P_f$.  Denote by $S_{\Gamma}$ the stratum in the Weil-Petersson boundary that corresponds to collapsing $\Gamma$ to a point, and let $S_{f^{-1}(\Gamma)}$ be the stratum corresponding to the essential pre-image of $\Gamma$.  Selinger showed that $\sigma_f(S_{\Gamma})\subset S_{f^{-1}(\Gamma)}$ \cite[p. 590]{Sel}.  We must then show that every such $\Gamma$ in $\rs\setminus P_f$ has an essential preimage.  Let $T_\Gamma$ be the Dehn twist with core curve $\Gamma$.  This is identified with a parabolic element of $\pi_1(\rs\setminus\Theta,0)$ whose cube can be lifted under $Y$ by examination of Figure \ref{fig:generatorlifts}.  Lemma \ref{lem:nonperipheral} demonstrates that the virtual endomorphism on moduli space maps parabolic elements to parabolic elements, and the image parabolic element is equivalent to a Dehn twist that fixes an essential curve that is precisely the essential component of $f^{-1}(\Gamma)$.  Therefore the Weil-Petersson boundary is mapped to itself.

\subsection{Dynamical behavior of $\phi_f$}
The motivation for the following theorem is that the dynamical behavior of $\phi_f$ applied to (powers of) parabolic elements describes the iterative boundary behavior of the extended Thurston pullback map.   For any $k\in\mathbb{Z}$ and $g\in G$, use the standard notation $g^{k\cdot w}=w^{-1}g^kw$ where one should recall that as a path, $w^{-1}$ is traversed first.  Though every parabolic element can be written in the form $\alpha^{n\cdot w}, \beta^{n\cdot w},$ or $\gamma^{n\cdot w}$, it becomes useful to consider words of the form $\delta^{n\cdot w}$ with $\delta=\beta^{-1}\alpha^{-1}$ to simplify many of the following statements.  Note that $\delta^{\beta^{-1}}=\gamma=\alpha^{-1}\beta^{-1}$.  
In this section the subscript on $\phi_f$ is suppressed.  Recall that $G$ denotes $\pi_1(\rs\setminus\Theta,0)$.  The word length of $g\in G$ with respect to the generating set $S=\{\alpha,\beta\}$ is denoted by $|g|$.

\begin{comment}
For an element $g^{k\cdot w}$ where $g\in\{\alpha,\beta,\gamma\},w \in G$ and $|g^{k\cdot w}|=|k|\cdot|g|+2|w|$, $w$ is said to be the \textit{conjugator} and $g$ is said to be the \textit{base}.
\end{comment}

\begin{theorem}  \label{thm:decreasingconjugators}
Let $w\in G$ and $x\in\{\alpha,\beta,\gamma\}$.  Then there is some $k\in\mathbb{N}$ and an appropriate choice of $n\in\mathbb{N}$ so that:
\[\phi^{\circ k}(x^{n\cdot w})\in\{x^m|m\in\mathbb{N}\}.\]

\end{theorem}

Put informally, iteration of $\phi$ will always eliminate the conjugator $w$ of a parabolic element. The first step in the proof is to write formulas that describe the effect of $\phi$ on parabolic elements.  Define the function $\overline{\phi}:\pi_1(\rs\setminus\Theta,0)\rightarrow\pi_1(\rs\setminus\Theta,0)$ as follows:

\begin{equation*}
\overline{\phi}(w) = \begin{cases}
\phi(w) & w\in H \\
\phi(\beta w) & w\in \beta^{-1}H\\
\phi(\alpha^{-1} w) & w \in \alpha H\\
\phi(\alpha w) & w \in \alpha^{-1}H\\
\end{cases}
\end{equation*}

\begin{lemma}\label{lem:phiFormulas}
For any $w\in G$ and $n\in\mathbb{Z}$, there exists $k\in\{1,3\}$ so that:

\begin{align}
\phi(\alpha^{3n\cdot w}) &= (\beta^{k\cdot n})^{\overline{\phi}(w)}\\\nonumber\\
\phi(\beta^{3n\cdot w}) &= \begin{cases}
(\alpha^{k\cdot n})^{\overline{\phi}(w)} & w\in H\cup\beta^{-1}H\cup\alpha H\\
(\alpha^{n})^{\beta^{-1}\overline{\phi}(w)} & w\in\alpha^{-1}H\\
\end{cases}\\\nonumber\\
\phi(\gamma^{3n\cdot w}) &= \begin{cases}
(\gamma^{k\cdot n})^{\phibar(w)} & w\in H\cup\alpha^{-1}H\\
(\delta^{n})^{\phibar(w)} & w\in\beta^{-1}H\cup\alpha H\\
\end{cases}\\\nonumber\\
\phi(\delta^{3n\cdot w})&=(\gamma^{k\cdot n})^{\phibar(w)}\\\nonumber
\end{align}
\end{lemma}
Schematically this lemma can be summarized by the following directed and labeled graph.  The vertices of the graph represent the base of the expressions in the lemma.  Directed edges indicate how the base changes under an application of $\phi$, and edges are labeled by the new exponent under application of $\phi$.\\

\centerline{\xymatrix{\alpha\ar@/^2pc/[rrr]^{\phibar(w)} &&&\beta\ar@/^.3pc/[lll]^{\phibar(w)}\ar@/^2pc/[lll]^{\beta^{-1}\phibar(w)} &&&{\gamma}\ar@(ul,dl)_{\phibar(w)} \ar@/^1pc/[rrr]^{\phibar(w)} &&&{\delta}\ar@/^1pc/[lll]^{\phibar(w)} } }

\begin{proof}
We prove the third equality for the case when $w\in \alpha^{-1}H$ and the others are proved analogously.  Let $w=\alpha^{-1}h, h\in H$.
\begin{eqnarray*}
\phi((\alpha^{-1}\beta^{-1})^{3n\cdot w}) 
&=&\phi(\alpha(\alpha^{-1}\beta^{-1})^{3n}\alpha^{-1})^{\phi(h)}\\
&=&(\alpha^{-1}\beta^{-1})^{3n\cdot\overline{\phi}(w)}
\end{eqnarray*}
\qed
\end{proof}

To prove the theorem, we show that $\phi$ has a particular kind of contracting property on the exponents of parabolic elements.  A cycle in the graph in Figure \ref{fig:VEmachine} which starts and ends at the vertex labeled 1 and passes through the vertex labeled 3, must immediately continue on to some vertex beside the one labeled 3.  Thus, a new directed labeled graph can be produced to condense the sequence of labels encountered along such paths.  This graph is exhibited in Figure \ref{fig:threeVertexMachine}, and is obtained from the graph in Figure \ref{fig:VEmachine} in the following way.  Delete the vertex labeled 3 and the four edges incident to it.
\begin{figure}[h]
\includegraphics[width=90mm]{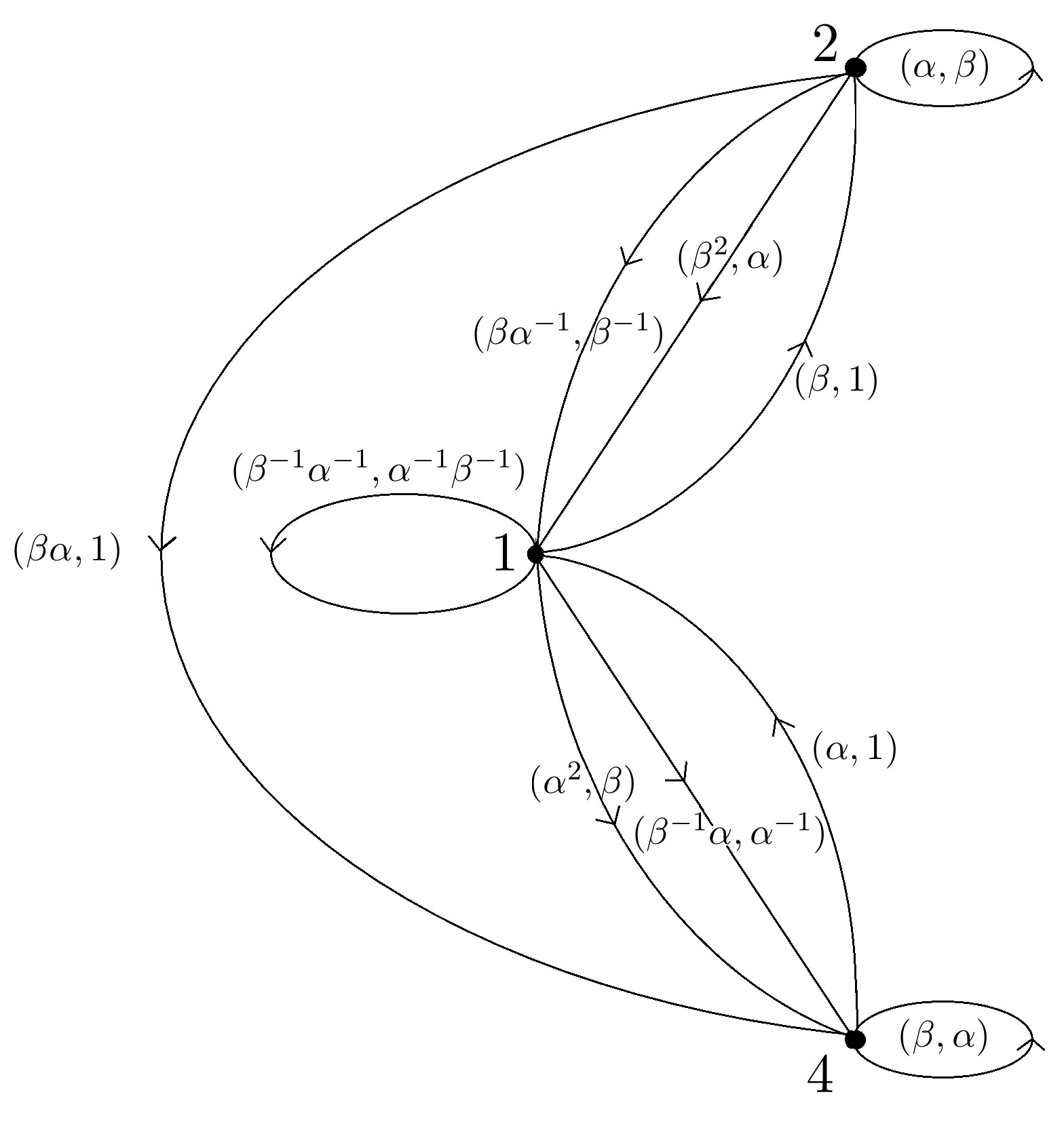}
\caption{New machine for computing $\phi_f$}
\label{fig:threeVertexMachine}
\end{figure}
Add six new edges which correspond to paths $e_1e_2$ in the old graph where $e_1$ and $e_2$ are either edges or reverse edges with the property that $t(e_1)$ is the vertex labelled 3 and $i(e_1)\neq t(e_2)$.  In the new graph, each such edge connects $i(e_1)$ to $t(e_2)$ and is given the label $\ell(e_1e_2)$ (the new edges corresponding to $e_1e_2$ and $e_2^{-1}e_1^{-1}$ are considered redundant and one of them is omitted).  One shows that for elements of $H$, this graph yields the same result as the graph in Figure \ref{fig:VEmachine} by computing its effect on the generators of $H$.  This new graph is the preferred perspective for the proof of the following lemma.

\begin{lemma}\label{lem:decreasingPhi}
Let $w\in H$.  Then either $|\phi(w)|\leq|w|-2$ or $|\phi(w)|=|w|$ in which case $w=(\alpha\beta)^k, k\in\mathbb{Z}$.
\end{lemma}

\begin{proof}
Elements of $H$ can be regarded as cycles in the Figure \ref{fig:threeVertexMachine} graph that begin and end at the vertex labeled 1.  Write $w=(\alpha\beta)^mh(\alpha\beta)^n$ where $h\in H$, $m,n\in\mathbb{Z}$ so that $|m|$ and $|n|$ are maximal, and $|w|=2|m|+|h|+2|n|$.  If $h$ is the identity, then evidently $|\phi(w)|=|w|$.  Otherwise $h$ corresponds to a path in Figure \ref{fig:threeVertexMachine} that begins and ends at the vertex labeled 1 whose first and last edges traversed are not $(\beta^{-1}\alpha^{-1},\alpha^{-1}\beta^{-1})$.  Note from Figure \ref{fig:threeVertexMachine} that the labels of the first and last edges that $h$ passes through decrease word length, and all other edges traversed by $h$ have labels that do not increase word length.
\qed
\end{proof}
Define the set of ``bad" elements $\mathcal{B}_{\phi}\subset H$ where $\phi$ is not length-decreasing: 
\[\mathcal{B}_{\phi}=\{(\alpha\beta)^k:k\in\mathbb{Z}\}.\]

\begin{lemma}\label{lem:phiBar}
Let $w\in G$.  Then precisely one of the following is true:
\begin{itemize}
\item $|\overline{\phi}(w)|\leq |w|-1$.
\item $|\overline{\phi}(w)|=|w|$ and $w=(\alpha\beta)^k$, $k\in\mathbb{Z}$.
\item $|\overline{\phi}(w)|=|w|+1$ and $w=\beta(\alpha\beta)^k$, $k\geq 0$.
\end{itemize}
\end{lemma}
Using suggestive notation, define $\mathcal{B}_{\overline{\phi}}$ by \[\mathcal{B}_{\phibar}=\mathcal{B}_{\phi}\cup\{\beta(\alpha\beta)^k, k\geq 0\}.\]
\begin{proof}
It is immediately evident from the definition of $\phibar$ and Lemma \ref{lem:decreasingPhi} that for any $w\in G$,
\[|\phibar(w)|\leq |w|+1.\]   
By direct computation we have 
\begin{equation}
\phibar((\alpha\beta)^k)=(\beta\alpha)^k,k\in \mathbb{Z}\\ \label{eqn:badset}
\end{equation}
\begin{equation}
\phibar(\beta(\alpha\beta)^k)=(\beta\alpha)^{k+1},k\geq 0. \label{eqn:extrabadset}
\end{equation}

Next suppose that $w\notin \mathcal{B}_{\phibar}$.  There are four subcases depending on the coset respresentative of $w$.  If $w\in H$, it is clear that $|\phibar(w)|=|\phi(w)|\leq|w|-2$.  To deal with the other three subcases, let $x\in\{\beta,\alpha,\alpha^{-1}\}$ be the inverse of the coset representative of $w$.  We must consider whether the process of appending the inverse of the coset representative and then reducing either lengthens or shortens the word.  First suppose that the length is shortened, i.e. $|xw|=|w|-1$.  Then
\[|\phibar(w)|=|\phi(xw)|\leq|xw|\leq|w|-1\]
If the length increases when $x$ is appended, i.e. $|xw|=|w|+1$, we must consider whether $xw$ is in the bad set or not.  If $xw\notin\mathcal{B_{\phi}}$, it is easy to see that $|\phibar(w)|\leq|w|-1$.  It is actually impossible for $xw\in\mathcal{B_{\phi}}$, because then $x=\alpha$ and so $w=\beta(\alpha\beta)^{k-1}$ contrary to the assumption that $w\notin \mathcal{B}_{\phibar}$.

\qed
\end{proof}

Continuing to the proof of the theorem, it will be helpful to note the following fact which can easily be verified by examination of the Schreier graph: if $k\equiv 0 \text{ mod }3$ then $(\beta\alpha)^k\in H$, if $k\equiv 1$ then $(\beta\alpha)^k\in\beta^{-1}H$, and if $k\equiv 2$ then $(\beta\alpha)^k\in\alpha H$.

\begin{proof}[Proof of Theorem \ref{thm:decreasingconjugators}]
To minimize notation in the following computations, denote by $*$ the presence of some integer that is necessary for each of the following expressions to be in $H$, though the precise values is not significant for our present concerns.  The value of $*$ may even vary within a single equation.  The directed graph from Lemma \ref{lem:phiFormulas} and the contracting property of Lemma \ref{lem:phiBar} make it evident that apart from a single exception, one application of $\phi$ to a suitable power of a parabolic word with exponent $w\notin\extrabadset$ will decrease the length of the exponent.  This single exception occurs when the base is $\beta$ and $w\in\alpha^{-1}H$ in which case the length of the exponent may be preserved; but if the new exponent $\beta^{-1}\phibar(w)$ lies outside of $\extrabadset$, the next iterate of $\phi$ will decrease its length according to Lemma \ref{lem:phiBar}.  The case when the new exponent lies in $\extrabadset$ is the only thing case left to consider.  

When $w\in\extrabadset$, there are four cases that must be considered to show that the exponent lengths decrease under application of $\phi$ to the parabolic element.  For $x\in\{\alpha,\beta\}$, it is shown that $\phi^{\circ 2}(x^{*w})=x^{*w'}$ for some minimal $w'$, and $|w'|<|w|$.  For $x\in\{\gamma,\delta\}$, one can show that $\phi(x^{*w})=\gamma^*$, but this is omitted.\\

\textbf{Case when base is $\alpha$:}  First assume that $w\in\badset$, and consider separately the situation when $k>0$ and $k<0$ for $w=(\alpha\beta)^k$.  If $k<0$, note that 
\[\phi(\alpha^{*w})\stackrel{(1)}{=}\beta^{*\phibar((\alpha\beta)^k)}\stackrel{(\ref{eqn:badset})}{=}\beta^{*(\beta\alpha)^k}\]
and upon taking a second iterate we obtain:
\[\phi^{2}(\alpha^{*w})=\phi(\beta^{*(\beta\alpha)^k})\stackrel{(2)}{=}\alpha^{*\phibar((\beta\alpha)^k)}\]
where $|\phibar((\beta\alpha)^k)|<|w|$ because $(\beta\alpha)^k\notin\mathcal{B}_{\phibar}$.  When $k>0$ one observes that $\alpha^{*(\alpha\beta)^k}=\alpha^{*\beta(\alpha\beta)^{k-1}}$, which means that a simple cancellation puts the exponent in $\extrabadset\setminus\badset$.  

Finally, suppose $w\in\extrabadset\setminus\badset$.  Then
\[\phi(\alpha^{*w})\stackrel{(1)}{=}\beta^{*\phibar(\beta(\alpha\beta)^k)}\stackrel{(\ref{eqn:extrabadset})}{=}\beta^{*(\beta\alpha)^{k+1}}=\beta^{*\alpha(\beta\alpha)^k}\]
and taking a second iterate,
\[\phi^2(\alpha^{*w})=\alpha^{*\phibar(\alpha(\beta\alpha)^k)}\]
where $|\phibar(\alpha(\beta\alpha)^k)|<|w|$ since $\alpha(\beta\alpha)^k\notin\mathcal{B}_{\phibar}$.\\

\textbf{Case when base is $\beta$:}  If $w\in\badset$ we can assume that $k>0$, for otherwise a cancellation occurs.  So assuming $w=(\alpha\beta)^k,k>0$, it follows from Lemma \ref{lem:phiFormulas} that $\phi(\beta^{*w})\stackrel{(\ref{eqn:badset})}{=}\alpha^{*(\beta\alpha)^k}$ and so $\phi^{\circ 2}(\beta^{*w})\stackrel{(2)}{=}\beta^{*\phibar((\beta\alpha)^k)}$, where one sees that $|\phibar((\beta\alpha)^k)|<|w|$ since $(\beta\alpha)^k\notin\extrabadset$.  The second case is when $w\in\extrabadset\setminus\badset$, but one immediately sees that a cancellation with the base occurs that puts the exponent in $\badset$.
\qed
\end{proof}

\end{comment}

\section{Properties of $\sigma_f:\Qbar\rightarrow\Qbar$}
\label{SigmaProps}

In earlier sections, the Weil-Petersson boundary of Teichm\"uller space was identified with the extended rationals $\Qbar$, and the observation was made that the Thurston pullback map extends to the boundary which maps to itself.  Denote by $\sigma_f:\Qbar\rightarrow\Qbar$ the restriction of the extended pullback map to the Weil-Petersson boundary under the identification with $\Qbar$.

The following functional equation is a crucial computational tool, where $\frac{p}{q}\in\Qbar$ and $w\in H$:
\begin{equation}
\label{eqn:functionalEquation}
\sigma_f(\frac{p}{q}.w)=\sigma_f(\frac{p}{q}).\phi(w).
\end{equation}
This equation appears in \cite{BN06} and is a consequence of the following commutative diagram and the fact that $\sigma_f$ and the action of PMCG$(\rs,P_f)$ extend continuously to the Weil-Petersson boundary.  Denote by $T_{p/q}$ the right Dehn twist that fixes the point $\frac{p}{q}$ in the Weil-Petersson boundary and denote by $T_w$ the mapping class that comes from pushing $0$ along the positive direction of $w\in H$.

\centerline{ \xymatrix{(\rs,P_f) \ar[r]^{T_{\sigma_f(p/q)}} \ar[d]_f
&{(\rs,P_f)} \ar[d]^{f} \ar[r]^{T_{\phi(w)}} &(\rs,P_f) \ar[d]^f \\ (\rs,P_f) \ar[r]^{T_{p/q}}
&{(\rs,P_f)} \ar[r]^{T_w} &(\rs,P_f) }}

For future reference it is necessary to state some simple results, the first being that $\sigma_f(\frac{1}{1})=-\frac{1}{1}$.  This is demonstrated as follows:
\[
\sigma_f(\frac{1}{1})=\sigma_f(\frac{1}{1}.\beta^{-1}\alpha^{-1})=\sigma_f(\frac{1}{1}).\phi(\beta^{-1}\alpha^{-1})=\sigma_f(\frac{1}{1}).\alpha^{-1}\beta^{-1}
\]
where $\sigma_f(\frac{1}{1})=-\frac{1}{1}$ because it is the fixed point of the action of $\alpha^{-1}\beta^{-1}$.
Similarly,
\[\sigma_f(-\frac{2}{1})=\frac{1}{0}\]
\[\sigma_f(-\frac{1}{2})=\sigma_f(\frac{1}{2})=\frac{0}{1}\]
\begin{equation}
\label{eqn:oneOverOne}
\sigma_f(\frac{1}{1})=\sigma_f(\frac{1}{3})=-\frac{1}{1}
\end{equation}

\noindent The following consequence of Theorem \ref{thm:decreasingconjugators} describes the global dynamics of $\sigma_f:\Qbar\rightarrow\Qbar$:
\begin{theorem}
\label{theorem:explicitFGA}
Let $\frac{p}{q}\in\Qbar$ be a reduced fraction.  Then under iteration of $\sigma_f$, $\frac{p}{q}$ lands either on the two-cycle $\frac{0}{1}\leftrightarrow\frac{1}{0}$ or on the fixed point $-\frac{1}{1}$.  More precisely, $\frac{p}{q}$ lands on $-\frac{1}{1}$ if and only if $p$ and $q$ are odd. 
\end{theorem}

\begin{proof}
Recall that points in the Weil-Petersson boundary are encoded by Dehn twists $T_{p/q}$.  It is a consequence of Equation \ref{eqn:functionalEquation} that $\sigma_f(\frac{p}{q})=\frac{p'}{q'}$ if and only if $\frac{p'}{q'}=\text{Fix}(\phi(T_{p/q}^k))$ for some appropriate value of $k\in\mathbb{N}$ since
\[\sigma_f(\frac{p}{q})=\frac{p'}{q'}\iff \sigma_f(\frac{p}{q}.T_{p/q}^k)=\frac{p'}{q'} \iff \sigma_f(\frac{p}{q}).\phi(T_{p/q}^k)=\frac{p'}{q'}. \]
For some choice of $k\in\mathbb{N}$, Theorem \ref{thm:decreasingconjugators} claims that under iteration of $\phi$, $T_{p/q}^k$ lands in one of the three maximal parabolic subgroups $\langle\alpha\rangle,\langle\beta\rangle,$ or $\langle\gamma\rangle$. Then since $\phi(\alpha^3)=\beta,\phi(\beta^3)=\alpha$ and $\phi(\gamma^3)=\gamma$, the mapping 
properties of $\sigma_f$ on the global attractor are easily found.

The continued fraction algorithm presented before gives a way of writing $\frac{p}{q}=*.w$ where $*\in\{\frac{0}{1},\frac{1}{0},\frac{1}{1}\}$ and $w\in G$.  In the case where $*=\frac{1}{1}$ and depending on the coset of $H$ containing $w$, Equations \ref{eqn:functionalEquation} and \ref{eqn:oneOverOne} imply:
\begin{align*}
w &\in H, \text{ then }  \sigma_f(\frac{p}{q})=\sigma_f(\frac{1}{1}.w)=-\frac{1}{1}.\phi(w).\\
w &\in \alpha H, \text{ then } \sigma_f(\frac{p}{q})=\sigma_f(-\frac{1}{1}.\alpha^{-1}w)=-\frac{1}{1}.\phi(\alpha^{-1}w).\\
w &\in \alpha^{-1} H, \text{ then } \sigma_f(\frac{p}{q})=\sigma_f(\frac{1}{3}.\alpha w)=-\frac{1}{1}.\phi(\alpha w).\\
w &\in \beta^{-1} H, \text{ then } \sigma_f(\frac{p}{q})=\sigma_f(-\frac{1}{1}.\beta w)=-\frac{1}{1}.\phi(\beta w).
\end{align*}
Similar computations apply when $*=\frac{0}{1}$ and $*=\frac{1}{0}$, and so the following hold for all $w\in G$:
\[\sigma_f(\frac{1}{0}.w)=\frac{0}{1}.\phibar(w)\]
\[\sigma_f(\frac{0}{1}.w)=\frac{1}{0}.\phibar(w)\]
\[\sigma_f(\frac{1}{1}.w)=-\frac{1}{1}.\phibar(w).\]
Since the action of P$\Gamma(2)$ on $\overline{\mathbb{Q}}$ preserves the parity of numerator and denominator, these equations make evident that when $p$ and $q$ are both odd, then $\sigma_f(\frac{p}{q})$ will have odd numerator and denominator in reduced form.  One can also conclude that when $p$ is odd and $q$ is even, then $\sigma_f(\frac{p}{q})$ will have \textit{even} numerator and \textit{odd} denominator when written in reduced form.  An analogous result holds for $p$ even and $q$ odd.
\qed
\end{proof}

The following sample orbits of $\sigma_f$ demonstrate that a fraction with odd numerator and even denominator can land on either $\frac{0}{1}$ or $\frac{1}{0}$:

\begin{align*}
\frac{203}{356}&\longmapsto -\frac{50}{33}\longmapsto -\frac{13}{6}\longmapsto \frac{6}{1}\longmapsto -\frac{1}{2}\longmapsto \frac{0}{1}\\
\frac{203}{354}&\longmapsto -\frac{28}{19}\longmapsto -\frac{7}{4} \longmapsto -\frac{4}{1}\longmapsto \frac{1}{0}.\\
\end{align*}

Next we show that $\sigma_f$ is surjective and that all fibers are infinite.  Recall that $\alpha$ and $\beta$ generate $G$.  Observe that $\phi$ is a surjective virtual endomorphism since $\phi(\alpha^{-1}\beta\alpha)=\alpha$ and $\phi(\beta\alpha\beta^{-1})=\beta$.  Thus, if $\frac{p}{q}=\frac{0}{1}.w'$ where $w'\in G$, then $\sigma_f(\frac{1}{0}.w)=\frac{p}{q}$ where $w$ is chosen so that $\phi(w)=w'$.  We show that $\sigma_f$ is infinite-to-one in the case of rational numbers of the form $\frac{1}{0}.w'$.  Since $\phi(\beta^2\alpha^2)$ is trivial, one knows from Equation \ref{eqn:functionalEquation} that $\sigma_f(\frac{0}{1}.(\beta^2\alpha^2)^k)=\frac{1}{0}$ for all $k\in\mathbb{Z}$.  Since $\{\frac{0}{1}.(\beta^2\alpha^2)^k:k\in\mathbb{Z}\}$ is infinite, the preimage of $\frac{1}{0}$ is infinite.  By surjectivity of $\phi$, there is a $w$ so that $\phi(w)=w'$ and the infinite set of fractions $\{\frac{0}{1}.(\beta^2\alpha^2)^kw:k\in\mathbb{Z}\}$ all map to $\frac{p}{q}$.

\begin{figure}[h]
\centerline{\includegraphics[width=130mm]{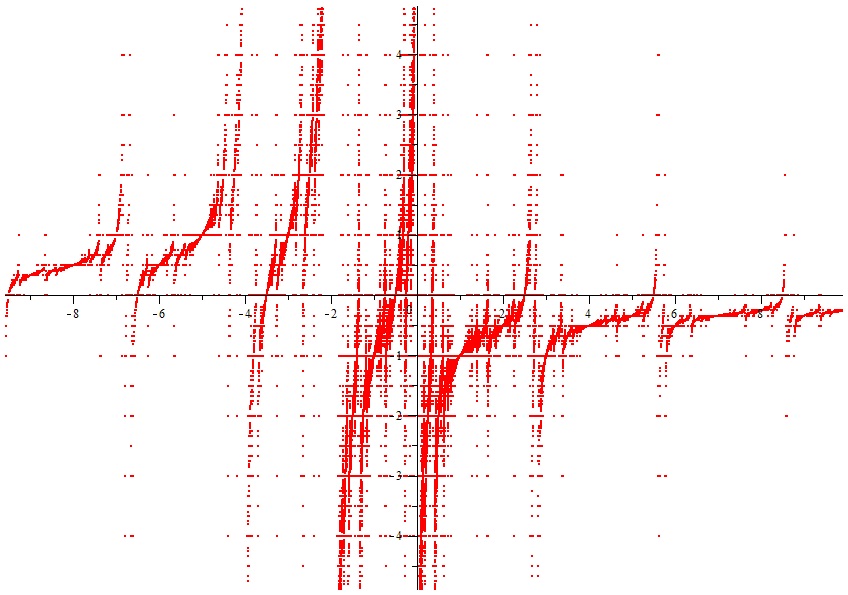}}
\caption{A portion of the plot of the points $(\frac{p}{q},\sigma_f(\frac{p}{q}))$ with max$(|p|,|q|)<1000$}
\label{fig:sigmaPlot}
\end{figure}

Three identities are given that explain some behavior of the graph of $\sigma_f$ in Figure \ref{fig:sigmaPlot}.  There are numerous other identities which can be proven similarly.
Using Equation $\ref{eqn:functionalEquation}$ and the fact that $\phi((\alpha\beta)^n)=(\beta\alpha)^n$,
\[\sigma_f(\frac{n+1}{n})=-\frac{n}{n+1},n>0\]
\begin{equation*}
\sigma_f(\frac{n}{n+1}) = \begin{cases}
-\frac{n-1}{n-2} & n>0 \text{ odd} \\
-\frac{n+1}{n} & n>0 \text{ even.}\\
\end{cases}
\end{equation*}
Another useful identity is proven using the real symmetry of $f$:
\begin{equation}
\label{eqn:inverse}
\sigma_f((\frac{p}{q})^{-1})=(\sigma_f(\frac{p}{q}))^{-1}.
\end{equation}

\bibliographystyle{plain}
\bibliography{Lodge}

%============Personal information==========================================

\end{document}